\documentclass[aap]{imsart}

\RequirePackage{amsthm,amsmath,amsfonts,amssymb}
\RequirePackage[numbers]{natbib}
\RequirePackage[colorlinks,citecolor=blue,urlcolor=blue]{hyperref}
\RequirePackage{graphicx}

\RequirePackage{cleveref}
\usepackage{enumitem}
\setlist[enumerate]{leftmargin=.5in}
\setlist[itemize]{leftmargin=.5in}

\RequirePackage{dsfont}
\RequirePackage{subfig}

\startlocaldefs
\theoremstyle{plain}
\newtheorem{lemma}{Lemma}
\newtheorem{proposition}{Proposition}
\newtheorem{theorem}{Theorem}
\newtheorem{corollary}{Corollary}
\newtheorem{conjecture}{Conjecture}
\theoremstyle{remark}
\newtheorem{remark}{Remark}
\newtheorem{example}{Example}
\newtheorem{definition}{Definition}

\def\diag{\operatorname{diag}}
\def\V#1{{\mathbf #1}}

\def\O{\operatorname{O}}
\def\P{\mathbb P}
\def\E{\mathbb E}
\def\U{\operatorname{U}}
\def\SO{\operatorname{SO}}
\def\SU{\operatorname{SU}}
\def\B{\operatorname{B}}
\def\GL{\operatorname{GL}}

\def\Haar{\operatorname{Haar}}
\def\Bernoulli{\operatorname{Bernoulli}}
\def\Uniform{\operatorname{Uniform}}

\def\t{\theta}

\endlocaldefs

\begin{document}

\begin{frontmatter}
\title{Distribution of the number of pivots needed using Gaussian elimination with partial pivoting on random matrices}
\runtitle{Distribution of the number of pivots needed using GEPP}


\begin{aug}
\author[A]{\fnms{John}~\snm{Peca-Medlin}\ead[label=e1]{johnpeca@math.arizona.edu}}
\address[A]{Department of Mathematics, University of Arizona\printead[presep={,\ }]{e1}}

\end{aug}

\begin{abstract}
Gaussian elimination with partial pivoting (GEPP) {is a widely used} method to solve dense linear systems. Each GEPP step uses a row transposition pivot movement if needed to ensure the leading pivot entry is maximal in  magnitude for the leading column of the remaining untriangularized subsystem. We will use theoretical and numerical approaches to study how often this pivot movement is needed. We provide full distributional descriptions for the number of pivot movements needed using GEPP using particular Haar random ensembles, as well as compare these models to other common transformations from randomized numerical linear algebra. Additionally, we introduce new random ensembles with fixed pivot movement counts and fixed sparsity, $\alpha$. Experiments estimating the empirical spectral density (ESD) of these random ensembles leads to a new conjecture on a universality class of random matrices with fixed sparsity whose scaled ESD converges to a measure on the complex unit disk that depends on $\alpha$ and is an interpolation of the uniform measure on the unit disk and the Dirac measure at the origin.
\end{abstract}



\end{frontmatter}


\section{Introduction and background}
\label{sec:intro}

Gaussian elimination (GE) is {a commonly used} method to solve linear systems 
\begin{equation}
\label{eq:linear}
A\V x = \V b
\end{equation}
for $A \in \mathbb R^{n\times n}$, and remains a staple of introductory linear algebra courses.  If no leading principal minors of $A$ vanish, then GE iteratively transforms \eqref{eq:linear} into two equivalent triangular systems, resulting in the factorization $A = LU$ for $L$ unipotent lower triangular and $U$ upper triangular matrices using {$\frac23n^3(1+ o(1))$} FLOPs.  If $A$ is nonsingular but does have a vanishing principal minor, then GE would need to be combined with a selected pivoting strategy to ensure GE can be continued at each intermediate step without encountering a zero pivot. The row and column movements used to ensure nonzero pivots would then result in additional permutation matrices $P,Q$ for a modified GE factorization $PAQ=LU$. Even when pivoting is not necessary, {pivoting} remains desirable for added computational stability for certain pivoting strategies. This includes GE with partial pivoting (GEPP), the most prominent pivoting strategy for dense linear systems, which is the default strategy used by MATLAB with its built-in \texttt{lu} function. GEPP uses only row permutations and so results in a final $PA=LU$ factorization. (See \Cref{subsec:ge} for relevant description of GE, while \Cref{sec:perms} provides further background for GEPP.) 

With high performance computing, choosing a desired pivoting strategy with GE often becomes a balancing act that takes into account the total computation time (viz., total FLOP count) rather than just accuracy (viz., the numerical stability of computed solutions). The cost of moving large amounts of data can be very expensive on high performance machines. For example, numerical experiments on a hybrid CPU/GPU setup have shown GEPP used with moderately sized random matrices have pivoting account for over 20 percent of the total computation time \cite{baboulin}. Hence, limiting pivot movements using GEPP is desirable to save time. Parker introduced a preconditioning method through the use of random butterfly matrices to remove the need for pivoting overall for any nonsingular linear system \cite{Pa95}. Butterfly matrices are a recursively defined class of orthogonal matrices (see \Cref{subsec:butterfly} for a full definition of random butterfly matrices) for which matrix-vector multiplication $A \V x$ is computed using $3n \log_2 n$ FLOPs rather than the $\mathcal O(n^2)$ FLOPs needed using a general dense matrix. Parker established for $U,V$ independent and identically distributed (iid) random butterfly matrices, then $A \V x = \V b$ can be transformed into the  equivalent system 
\begin{equation}
\label{eq:linear butterfly}
UA\V y = U\V b \quad \mbox{and} \quad \V x = V^* \V y
\end{equation}
for which GE without pivoting (GENP) can be carried out with probability near 1. The above computations show then transforming \eqref{eq:linear} into \eqref{eq:linear butterfly} can be performed using $\mathcal O(n^2 \log_2 n)$ FLOPs, and hence does not impact the leading order complexity of GE. 

In \cite{jpm}, Peca-Medlin and Trogdon further explored the numerical stability of using GE with a variety of pivoting strategies in addition to using randomized preconditioning methods, which included random butterfly  and Haar orthogonal transformations. One output was certain preconditioners had the impact 
that running the preconditioner followed then by GE with another pivoting strategy could ``upgrade'' the pivoting strategy in terms of the numerical accuracy for the computed solution. For instance, even though GENP often leads to accuracy far from that achieved using GEPP or GE with complete pivoting (GECP), a preconditioned matrix using GEPP would lead to accuracy on par with using GECP. Adding one step of iterative refinement would further seal this alignment in accuracy. 

A natural direction that arose out of this previous analysis in \cite{jpm} was to better understand how many actual pivot movements are needed with these different pivoting strategies on the preconditioned linear systems. The goal of this paper is to provide a clearer answer to this question with respect to GEPP. GEPP can use a pivot movement at each GE step, for up to $n-1$ total pivot movements. So applying GEPP to a random linear system will result in the number of pivot movements being a random variable with support in $0,1,\ldots,n-1$.  We will study the question of how much movement one should expect if they choose to use GEPP in conjunction with randomized preconditioning methods\footnote{This is a simpler focus than the more general study of the distribution of the GEPP permutation matrix factor.}.

Our results include both theoretical and numerical approaches focusing on applying GEPP to several random matrix ensembles. The theoretical results rely on input matrices from Haar orthogonal and butterfly ensembles (see \Cref{subsec:prelim} for a review of Haar measures). Our numerical studies  further study these models in relation to other common transformations from randomized numerical linear algebra transformations, which expand studies in \cite{jpm}. 

\subsection{Outline of paper}

This paper is structured to explore the question of how much actual pivot movement is necessary when using GEPP with a variety of randomized preconditioning methods on particular nonsingular linear systems. {The appendices  introduce  necessary notation and background used throughout the document (\Cref{subsec:prelim}); preliminary background on GE, GEPP and growth factors (\Cref{subsec:ge}); definition and background for the Stirling-1 distribution, which includes connections to previous statistical results of distributions using Stirling numbers of the first kind (\Cref{sec:stirling}); and definitions and recent results involving butterfly matrices and butterfly permutations (\Cref{subsec:butterfly}).}

\Cref{sec:results} provides a statement of the main theoretical result, \Cref{thm:main}, that provides the full distribution for the number of GEPP pivot movements needed for particular random ensembles. This includes the distribution of pivot movements when 
\begin{enumerate}[label=\Roman*.]
\item the resulting GEPP permutation matrix factor $P$ is uniformly distributed among the permutation matrices, which uses the Stirling-1 distribution, $\Upsilon_n$, that satisfies
\begin{equation}
\label{eq:stirling pmf}
\P(\Upsilon_n = k) = \frac{|s(n,k)|}{n!}
\end{equation}
for $k=1,2,\ldots,n$ where $s(n,k)$ is the Stirling number of the first kind; and 
\item when GEPP is applied to a Haar-butterfly matrix, which results in a scaled Bernoulli {random variable}. 
\end{enumerate}
The remainder of \Cref{sec:results} provides results and implications of (I) in connection with $QR$ factorizations and other standard Haar unitary models. The proof of \Cref{thm:main} is postponed {until} \Cref{sec:perms,sec:butterfly}.

\Cref{sec:perms} provides the necessary background to establish a proof of part (I) of \Cref{thm:main}. This includes introducing explicit decompositions of permutations in $S_n$, the group of permutations of $n$ objects, that connect explicitly to GEPP permutation matrix factors as well as uniform sampling of $S_n$. \Cref{sec:butterfly} provides a more thorough background for $\mathcal P_N^{(B)}$, the butterfly permutation matrices, and yields a proof for part (II) of \Cref{thm:main}. Additionally, explicit structural configurations of exact pivot movement locations of Haar-butterfly matrices are established that yield a distribution on the pivot location configurations.

\Cref{sec:new} builds on top of \Cref{thm:main} to introduce a new ensemble of random matrices that align with uniformly sampling from $\GL_n(\mathbb R)/\mathcal U_n(\mathbb R)$, the left cosets of the group of nonsingular upper triangular matrices in the general linear group. This ensemble can be {adjusted to} sample from random ensembles with fixed pivot movement distributions. General random ensembles are {then} introduced with fixed sparsity conditions{, and a} new conjecture is provided for the asymptotic empirical spectral distribution for this generalized random ensemble with fixed sparsity that connects to and subsumes the famous circular law in random matrix theory.

\Cref{sec:num} uses numerical experiments to further explore the distribution of pivot movements needed using other transformations from randomized numerical linear algebra. These experiments focus on three initial models, two that need minimal GEPP pivot movements and one that requires the maximal number of GEPP pivot movements. These approaches build on top of the numerical experiments used in \cite{jpm}, as well as connect to other random models used in earlier sections.

\section{Distribution of GEPP pivot movements for particular random matrices}
\label{sec:results}

We first  state the main theoretical result on the distribution of the number of GEPP pivot movements used that applies to particular input random matrix models. These will make use of $\Upsilon_n$, the Stirling-1 distribution (see \Cref{sec:stirling}), $\B_s(N,\Sigma_S)$, the Haar-butterfly matrices, and $\mathcal P_N^{(B)}$, the butterfly permutation matrices (see \Cref{subsec:butterfly}).

\begin{definition}
    {Let $\Pi(A) \in \{0,1,\ldots,n-1\}$ denote the number of pivots needed using GEPP for a $n\times n$ matrix $A$.}
\end{definition}


\begin{theorem}
\label{thm:main}
\normalfont{(I)} If $A$ is a $n\times n$ random matrix with independent columns whose first $n-1$ columns have (absolutely) continuous iid entries, then $P \sim \Uniform(\mathcal P_n)$ for $PA = LU$ the GEPP factorization of $A$ for $n \ge 2$ and {$\Pi(A) \sim n - \Upsilon_n$}.

\normalfont{(II)} If $B \sim \B_s(N,\Sigma_S)$, then $P \sim \Uniform(\mathcal P_N^{(B)})$ for $PB = LU$ the GEPP factorization, and {$\Pi(B) \sim \frac{N}2 \Bernoulli\left(1-\frac1N\right)$}.
\end{theorem}
The proof of \Cref{thm:main} will be postponed until \Cref{sec:perms} for (I) and \Cref{sec:butterfly} for (II).

Part (I) gives the most general result that yields pivot movements following the Stirling-1 distribution. This includes iid random matrices when the entry distribution is continuous.

\begin{corollary}
\label{cor:iid}
If $A$ is a $n\times n$ matrix with (absolutely) continuous iid entries, then {$\Pi(A) \sim n - \Upsilon_n$}.
\end{corollary}

If $A$ is an iid matrix with entries taken from a distribution $\xi$ with atoms (i.e., $\P(\xi = c) > 0$ for some $c$), then GEPP would yield a different distribution on {$\Pi(A)$}. 

\begin{example} 
\label{ex:bern}
If $A$ is $2\times 2$ where $A_{ij}$ are each iid $\operatorname{Bernoulli}(p)$, then GEPP yields the permutation matrix factor $P = P_{(1 \ 2)}^\zeta$ where $\zeta \sim \Bernoulli(p(1-p))$. This follows since a pivot movement is needed only if $A_{11} = 0$ and $A_{21} = 1$. A pivot is needed with probability $p(1-p) \le \frac14$, so {$\Pi(A) \sim \zeta$}.
\end{example}

Other continuous random models that do not fit the {independent column condition} in (I) would yield different distributions on the resulting GEPP permutation matrix factors. 

\begin{example}
\label{ex:goe}
Consider $G \sim \operatorname{GOE}(2)$, where $G_{11},G_{22} \sim N(0,2)$ and $G_{21}=G_{12} \sim N(0,1)$ are independent. Then {$\Pi(G) \sim \Bernoulli(p)$} for
\begin{align*}
    p &= \P(|G_{21}| > |G_{11}|) = \P(|Z_1| > \sqrt 2 |Z_2|) 
    = \P(Z_1^2/Z_2^2 > 2) \\
    &= \P(F_{1,1} > 2) =\frac1\pi \int_2^\infty \frac{d x}{\sqrt x(1+x)} = \frac2\pi \arctan\left(\frac1{\sqrt 2}\right) \approx 0.391826552
\end{align*}
using $Z_i \sim N(0,1)$ iid and $F_{\mu,\nu}$ to denote the $F$-distribution with $\mu>0$ numerator degrees of freedom (d.f.) and $\nu>0$ denominator d.f. 
\end{example}

\begin{example}
\label{ex:gue}
Consider now $G \sim \operatorname{GUE}(2)$, where  $G_{11},G_{22} \sim N(0,1)$ while $G_{12}=\overline{G_{21}} \sim N_{\mathbb C}(0,1)$, then {similarly $\Pi(G) \sim \Bernoulli(q)$} for
\begin{align*}
    q &= \P(|G_{21}| > |G_{11}|) = \P\left(\sqrt{(Z_1^2+Z_2^2)/2} > |Z_3|\right)=\P((Z_1^2+Z_2^2)/2 > Z_3^2) \\
    &= \P(F_{2,1} > 1) = \int_1^\infty \frac{dx}{(1+2x)^{3/2}} = \frac1{\sqrt 3} \approx 0.577350269.
\end{align*}
\end{example}

\begin{remark}
In comparison to \Cref{ex:bern,ex:goe,ex:gue}, \Cref{cor:iid} yields if $G$ is a continuous iid $2\times 2$ matrix {(e.g., $G \sim \operatorname{Gin}(2,2)$)}, then {$\Pi(G) \sim 2 - \Upsilon_2 \sim \Bernoulli(\frac12)$}, where we note $\P(\Upsilon_2 = 1) = \frac{|s(2,1)|}{2!} = \frac12 = \frac{|s(2,2)|}{2!} =\P(\Upsilon_2 = 2)$.
\end{remark}

{Further results from \Cref{thm:main} and \Cref{cor:iid} can be gleaned by relating the $LU$ and $QR$ factorizations of a matrix.} If $A$ has the GEPP factorization $PA = LU$, then its $QR$ factorization $A = QR$ yields $PQ = AR^{-1} = L(UR^{-1})$.\footnote{Note $UR^{-1} \in \mathcal U(\mathbb F,n)$ when $U,R \in \mathcal U(\mathbb F,n)$ since $\mathcal U(\mathbb F,n)$ is a group.} In particular, the resulting permutation matrix and hence the pivot movements that would be needed using GEPP on $Q$ and $A$ are identical when no ties are encountered by \Cref{thm:GEPP_unique}{, i.e., $\Pi(Q) = \Pi(A)$}. This observation then can be combined with Stewart's realization of Haar orthogonal and Haar unitary sampling using Ginibre ensembles (cf. \Cref{subsec:prelim}) to yield:
\begin{corollary}
\label{cor:haar o}
If $A \sim \operatorname{Haar}(\O(n))$ or $A \sim \operatorname{Haar}(\U(n))$, then {$\Pi(A) \sim n - \Upsilon_n$}.
\end{corollary}

A similar approach can then yield information about the pivot movements needed on $\Haar(\SO(n))$ and $\Haar(\SU(n))$. Note Stewart's sampling method for $\Haar(\O(n))$ can be indirectly rephrased in terms of the Subgroup algorithm, which was formally established by Diaconis and Shahshahani \cite{DiSh87}. The Subgroup algorithm enables a uniform sampling method for a compact group by using a subgroup and its associated cosets: 
\begin{theorem}[Subgroup algorithm,\cite{DiSh87}]
\label{thm:subgroup}
If $G$ is a compact group, $H$ is a closed subgroup of $G$, and $H/G$ is the set of left-costs of $H$, then if $x \sim \Uniform(G/H)$ and $y \sim \Haar(H)$, then $xy \sim \Haar(G)$. 
\end{theorem}
\Cref{thm:subgroup} can analogously be stated using right cosets. 

Stewart's approach for sampling Haar orthogonal matrices can be realized in light of \Cref{thm:subgroup} by iteratively using Householder reflectors as coset representatives of the subgroup of orthogonal matrices whose first row and column have 1 in the leading diagonal and are zero elsewhere.\footnote{The Householder reflectors then can be uniformly sampled by choosing $\V x \in \mathbb S^{n-1}$ uniformly.} More directly, one can then realize $\Haar(\SO(n))$ by using the coset representatives of $\O(n)/\SO(n)$ of $D(x) = \diag(x,1,\ldots,1)$ for $x = \pm 1$: if $A \sim \Haar(\SO(n))$ and $x \sim \Uniform(\pm1)$, then $D(x) A \sim \Haar(\O(n))$.\footnote{Similarly $\Haar(\U(n))$ and $\Haar(\SU(n))$ can be related using the coset representatives $D(x)$ for $x \in \mathbb T$.} Moreover, group homomorphisms can be used to yield uniform measures on cosets as push forward measures of Haar measures. In particular, since $\SO(n)$ is a normal subgroup of $\O(n)$ (since it is the kernel of the group homomorphism $\det: \O(n) \to \mathbb R$), then the natural quotient map $C_2 = \{\pm 1\} \cong \O(n)/\SO(n)\cong \O(n)\backslash \SO(n)$ yields $C_2 \times \SO(n) \cong \O(n)$. This yields uniform sampling from both $C_2$ and $\SO(n)$ using push forwards of the Haar sampling on $\O(n)$ along with the natural projection maps to each component, which similarly holds using instead $\U(n)$ and $\SU(n)$: 
\begin{corollary}
$x \sim \Uniform(\pm 1)$ and $A \sim \Haar(\SO(n))$ iff $AD(x) \sim \Haar(\O(n))$, while $y \sim \Uniform(\mathbb T)$ and $A \sim \Haar(\SU(n))$ iff $AD(y) \sim \Haar(\U(n))$.
\end{corollary}

Hence, mimicking the result aligning the GEPP permutation matrix factors between $A$ and the $Q$ from the $A=QR$ factorization, we similarly have identical {GEPP permutation} $P$ factors for $A \sim \Haar(\O(n))$ and $B \sim \Haar(\SO(n))$ where $A = B D(x)$.\footnote{The argument from the paragraph before \Cref{cor:haar o} is identical after replacing $R \in \mathcal U(\mathbb F,n)$ with $D(x) \in \mathcal U(\mathbb F,n)$.}
\begin{corollary}
If $A \sim \operatorname{Haar}(\SO(n))$ or $A \sim \operatorname{Haar}(\SU(n))$, then {$\Pi(A) \sim n - \Upsilon_n$}.
\end{corollary}

In \cite{Pa95,jpm}, the authors studied particular random preconditioning transformations of the form $UAV^*$ for iid $U,V$ so that the resulting output  can almost surely (a.s.) have a $LU$ factorization. In \cite{jpm}, the authors further studied the numerical properties for GE pivoting strategies, including GEPP, after applying these random preconditioning transformations. Relating this to our current topic, one could ask how many pivot movements are needed using GEPP after these random transformations?

For the Haar orthogonal case (as well as the unitary, special orthogonal and special unitary cases), the result is immediate: Let $A$ be a nonsingular real matrix. Suppose $U,V$ are iid $\Haar(\O(n))$ and let $B = AV^*$, which is independent of $U$. Let $B = QR$ be the QR factorization of $B$. Then $UAV^* = UB = U(QR) = (UQ)R$. The permutation matrix factor resulting from GEPP for $UAV^*$ is then identical (if no ties are encountered, which holds a.s.) to the permutation matrix factor needed for $UQ$. Since $\Haar(\O(n))$ is right invariant, then $UQ \sim U$. This then yields {$\Pi(UAV^*) =\Pi(UQ) \sim \Pi(Q)$, which is  equal in distribution to case (I).}
\begin{corollary}
\label{cor:2sided}
If $A$ is a $n\times n$ nonsingular matrix, and $U,V$ are iid from the Haar measure on $\O(n)$, $\U(n)$, $\SO(n)$ or $\SU(n)$, then {$\Pi(UAV^*) \sim n - \Upsilon_n$}.
\end{corollary}

\begin{remark}
Extensions of (II) from \Cref{thm:main}  are less direct since (II) relies heavily on explicit structural properties of Haar-butterfly matrices, $\B_s(N,\Sigma_S)$. Analogous results can be established if the focus is restricted to matrices in $\displaystyle \bigotimes^n \mathbb R^{2\times 2}$.
\end{remark}

\section{Permutations from GEPP and uniform sampling of \texorpdfstring{{\boldmath$S_n$}}{2}}
\label{sec:perms}

Recall how the permutation matrix factor is formed when applying GEPP to a matrix. First a search for a largest magnitude element is performed on the first column of $A = A^{(1)}$. After that value is found, say at index $i_1$ (for $i_1 \ge 1$), then the corresponding row permutation for the transposition $\sigma_1 = (1 \ i_1)$ is applied to $A$ (using $P^{(1)} = P_{\sigma_1}$), after which the standard GE step follows to iteratively eliminate each element under the first diagonal value to form $A^{(2)}$ using the pivot element and the resulting lower triangular GE elimination factor {$ L^{(1,1)}$}, so that 
\begin{equation}
A^{(2)} = (L^{(1,1)})^{-1}P^{(1)}A^{(1)}.
\end{equation}
Then GE continues with a search for the largest magnitude element on the leading column of $A^{(2)}_{2:n,2:n}$, which results in a transposition of the form $\sigma_2 = (2 \ i_2)$ for $i_2 \ge 2$. The corresponding row permutation is performed (using $P^{(2)}$) followed by the standard GE elimination step using the second pivot (using {$ L^{(2,2)}$}), with which one groups the lower triangular  and permutation matrix factors using the relation $L^{(j,k)} = P^{(j)}L^{(j-1,k)}P^{(j)}$\footnote{Note $P^{(j)} = (P^{(j)})^{-1}$ since these permutation matrices correspond to transpositions that have order 2.}, so that
\begin{equation}
A^{(3)} = (L^{(2,2)})^{-1} P^{(2)} A^{(2)} = (L^{(2,1)}L^{(2,2)})^{-1}P^{(2)}P^{(1)}A^{(1)}.
\end{equation}
The process then continues, moving one column at a time, which results in the final GEPP factorization $PA = LU$ for {$P = P^{(n)}P^{(n-1)}\cdots P^{(2)}P^{(1)}$, $L = L^{(n,n)} \cdots L^{(n,2)}L^{(n,1)}$ and $U = A^{(n)}$, where necessarily $P^{(n)} = \V I$.}

Hence, the resulting permutation matrix factor is built up step by step using $\sigma_k = (k \ i_k)$, resulting in 
\begin{equation}
P = P^{(n-1)} \cdots P^{(2)}P^{(1)} = P_{\sigma_{n-1}} \cdots P_{\sigma_2} P_{\sigma_1} = P_{\sigma_{n-1} \cdots \sigma_2 \sigma_1} = P_\sigma
\end{equation}
for
\begin{equation}
\label{eq:perm_form}
    \sigma = (n-1 \ i_{n-1}) \cdots (2 \ i_2)(1 \ i_1)
\end{equation}
where $j \le i_j \le n$ for each $j$.
\begin{remark}
If $i_k = k$, then \eqref{eq:perm_form} can abstain from including the trivial permutation $(k \ i_k)$. In particular, \eqref{eq:perm_form} can trivially be expanded to $\sigma = (n \  i_n)\sigma$ where necessarily $i_n = n$.
\end{remark}

\eqref{eq:perm_form} is useful because every permutation can be put in this form.

\begin{lemma}
\label{lemma:perm}
Every permutation in $S_n$ can be written in the form \eqref{eq:perm_form}.
\end{lemma}
In particular, every permutation is realizable as corresponding to the GEPP permutation matrix factor for some input nonsingular matrix.  {This is equivalently explored in terms of permutation ``alignments'' in \cite{flajolet}. I will include a brief proof for completeness.}

\begin{proof}By counting, it is clear $n!$ inputs can be used to form $\sigma$ ($n$ choices for $i_1$, $n-1$ choices for $i_{n-1}$, and so on). Moreover, we see this correspondence is one-to-one: suppose $\sigma$ and $\sigma'$ are formed using distinct inputs, and let $k$ be the minimal index such that $i_k \ne i_k'$. Without loss of generality, assume $k = 1$. Let $\rho = \sigma(1 \ i_1)$ and $\rho' = \sigma'(1 \ i_1)$. Then $\rho(1) = \sigma(i_1)=1$ while $\rho'(1) = \sigma'(i_1) > 1$; it follows $\sigma \ne \sigma'$, which yields this  mapping is an injection and hence a bijection since $|S_n|$ is finite.

Alternatively, this can be realized through induction: this clearly holds for $n = 2$ since $S_2 = \{1,(1\ 2)\}$. Now assume it holds for $n-1$, then one can realize (using the inductive hypothesis) that $S_{n-1} \cong \{(n-1 \ i_{n-1})\cdots (2 \ i_2): j \le i_j \le n\}$ as a subgroup of $S_{n}$, by recognizing the permutations that fix 1 in $S_{n}$ is isomorphic to $S_{n-1}$. Moreover, adopting the crude labeling of $S_{n-1}$ for this subgroup of $S_{n}$, $[S_{n}:S_{n-1}] = n!/(n-1)! = n$, and every (right) coset representative for $S_{n-1}$ can be provided by $ S_{n-1}(1 \ i_1)$ so that $S_{n} = \bigsqcup_{j=1}^{n+1} S_{n-1} (1 \ j)$.
\end{proof}

The coset decomposition structure for $S_n$ used in the alternative proof of \Cref{lemma:perm} was utilized by Diaconis and Shashahani through an application of the Subgroup algorithm for generating a Haar measure on $S_n$ \cite{DiSh87}. Their result yields a means to sample uniformly from $S_{n}$ as the convolution of the Haar measure on $S_{n-1}$ and the uniform measure on $\{(1 \ j): 1 \le j \le n\}$. Explicitly, {one} can generate a uniform permutation $\sigma \in S_{n}$ by first uniformly (and independently) sampling both $\rho \in S_{n-1}$ and $\sigma_1 \in \{(1 \ j): 1\le j\le n\}$, and then forming $\sigma = \rho \sigma_1$. Iteratively applying this, one can uniformly sample a permutation by uniformly sampling each $i_k \in \{k,k+1,\ldots,n\}$, which yields a permutation in the form \eqref{eq:perm_form}. This establishes:
\begin{corollary}
\label{cor:uniform Sn}
If $i_k \sim \Uniform\{k,k+1,\ldots,n\}$ for each $k = 1,\ldots,n-1$, then 
\begin{equation}
(n-1 \ i_{n-1}) \cdots (2 \ i_2)(1 \ i_1) \sim \Uniform(S_n).
\end{equation}
\end{corollary}

Moreover, the steps where pivot movements occur can be read off directly from $\sigma$ when it is written in the form \eqref{eq:perm_form}. If $P_\sigma$ is the resulting permutation matrix factor from applying GEPP to $A$, then we explicitly know at step $k$ in GE, the pivot search on $A^{(k)}$ resulted in the transposition $(k \ i_k)$. This translates directly to how many pivot movements are needed through a particular implementation of GEPP by counting how many of these transpositions have $i_k > k$. (So no pivot movement is needed if $k = i_k$, since this translates to the leading pivot entry being maximal in its respective column.)

It follows then, explicitly, if $A$ is a random matrix such that the resulting GEPP permutation matrix factor $P$ satisfies $P \sim \Uniform(\mathcal P_n) = \Haar(\mathcal P_n)$\footnote{This is equivalent to the statement $P = P_\sigma$ for $\sigma \sim \Uniform(S_n)$.}, then {$\Pi(A)$}, the number of pivot movements needed during this implementation of GEPP on $A$,  necessarily satisfies 
\begin{equation}
\label{eq: num pivots}
    \P(\Pi(A) = k) = \frac{\#\{\sigma \in S_n: j = i_j \mbox{ for $n-k$ indices $j$}\}}{|S_n|}. 
\end{equation}
Furthermore, for $\sigma$ of the form \eqref{eq:perm_form}, the number of times $j > i_j$ then corresponds precisely to the number of disjoint cycles in the representation of $\sigma$: if $i_k = k$, then $k$ is contained in a cycle with elements no larger than $k$ (each successive row transposition will fix $k$); so $k$ is the maximal element in its cycle. (Recall if $k$ is a fixed point of a permutation, then $k$ comprises its own 1-cycle.) If $i_k > k$, then $k$ belongs to a cycle with a larger element. So counting the number of times $i_k = k$ is then equivalent to counting the number of disjoint cycles in the disjoint cycle decomposition of $\sigma$, since these can be enumerated by using their maximal elements. As is established in \Cref{sec:stirling}, these then align exactly with the absolute Stirling numbers of the first kind, $|s(n,k)|$.

\begin{example}
Consider $\sigma  = (5 \ 6)(3 \ 4)(2 \ 4)(1 \ 3)\in S_6$. By default $i_n = n$, so $i_6 = 6$ will be considered as not needing a pivot movement on the $n^{th}$ GE step (this is always vacuously true since GE completes after $n-1$ steps). Here, we have only $i_k = k$ for $k = 4$ and $k=6$, so we should expect $\sigma$ consists of precisely 2 cycles in its disjoint cycle decomposition, and 4 and 6 are the maximal elements in those two cycles. This is verified by 
finding the final disjoint cycle form of $\sigma = (1 \ 4\ 2 \ 3)(5 \ 6)$, for which $4$ and $6$ are the maximal elements of each of the disjoint cycles.
\end{example}

\begin{remark}
\label{rmk:max pivot}
Connecting the above conversation with the form \eqref{eq:perm_form}, one can form a $n$-cycle in $S_n$ by requiring $i_k > k$ for all $k$ up to $n-1$. Moreover, connecting this further to \Cref{cor:uniform Sn}, one can uniformly sample $n$-cycles by sampling $i_k \sim \Uniform\{k+1,\ldots,n\}$ for $k=1,2,\ldots,n-1$. Let
\begin{equation}
\label{eq:n-cycles}
\mathcal P_n^{n\operatorname{-cycles}} = \{P_\sigma: \mbox{$\sigma$ is a $n$-cycle in $S_n$}\}
\end{equation}
denote the subset of permutation matrices $\mathcal P_n$ that correspond to the $n$-cycles. If the corresponding GEPP permutation matrix factor is in $\mathcal P_n^{n\operatorname{-cycles}}$ for a $n\times n$ matrix $A$, then every GE step required a pivot movement for $A$. Moreover, if $i_k \sim \Uniform\{k+1,\ldots,n\}$ for each $k$ is used to generate the corresponding $n$-cycle $\sigma = (n-1 \ i_{n-1}) \cdots (1 \ i_1)$, then $P_\sigma \sim \Uniform(\mathcal P_n^{n\operatorname{-cycles}})$. 

\end{remark}

\subsection{Proof of (I) of \Cref{thm:main}}
Now we have the sufficient tools to establish:
\begin{quotation}\normalsize
\noindent \textsc{\Cref{thm:main}}: (I) \emph{If $A$ is a $n\times n$ random matrix with independent columns whose first $n-1$ columns have (absolutely) continuous iid entries, then $P \sim \Uniform(\mathcal P_n)$ for $PA = LU$ the GEPP factorization of $A$ for $n\ge 2$ and {$\Pi(A) \sim n - \Upsilon_n$}.}
\end{quotation}

\begin{proof}
Suppose $A$ satisfies the hypothesis {for (I)}, and let $P$ be the associated GEPP permutation matrix factor for $A$. We will prove $P \sim \Uniform(\mathcal P_n)$ using induction on $n \ge 2$. Suppose $n = 2$, so the first column of $A$ has continuous iid entries $A_{11}$ and $A_{21}$. Using GEPP, a pivot will be needed only if $|A_{21}| > |A_{11}|$, but $\P(|A_{11}| > |A_{21}|) = \frac12$ since $A_{11} \sim A_{21}$ (and $\P(|A_{11}| = |A_{21}|) = 0$ since these are continuous random variables). Hence, $P = P_{(1 \ 2)}^\zeta \sim \Haar(\mathcal P_2)$ for $\zeta \sim \Bernoulli(\frac12)$. 

Now assume the result holds for any random matrix of dimension $n-1$ with independent columns whose first $n-2$ columns have continuous iid entries. Let $A$ be the dimension $n$ matrix satisfying the statement. Using GEPP on $A$, for the first pivot search using the leading column of $A = A^{(1)}$, we have 
\begin{equation}
\P(\max(|A_{11}|,|A_{21}|,\ldots,|A_{n1}|) = |A_{11}|) = \P(\max(|A_{11}|,|A_{21}|,\ldots,|A_{n1}|) = |A_{k1}|)
\end{equation}
for each $k = 1,2,\ldots,n$ since $A_{11} \sim A_{k1}$ are continuous iid. It follows the first GEPP row transposition is of the form $P^{(1)} = P_{\sigma_1}$ for $\sigma_1 = (1 \ i_1)$ for $i_1 \sim \Uniform\{1,2,\ldots,n\}$, where $i_1 = \operatorname{argmax}_k |A_{k1}|$. Now applying the GE elimination factor $L^{(1,1)}$ results in $A^{(2)} = (L^{(1,1)})^{-1}P^{(1)}A^{(1)}$. Moreover, $\tilde A^{(1)} = P^{(1)}A^{(1)}$ still satisfies the (I) hypothesis since row permutations preserve each of the column independence and iid properties of $A^{(1)}$. $A^{(2)}$ then has entries of the form 
\begin{equation}
A^{(2)}_{ij} = \tilde A_{ij} - \tilde A_{1j} \cdot L^{(1)}_{i1} = \tilde A_{ij} - \tilde A_{1j} \cdot \frac{\tilde A_{i1}}{\tilde A_{11}}.
\end{equation}
for $i,j \ge 2$. In particular, since the columns are independent and have continuous iid entries, then $L^{(1)}_{i1}$ is independent of $\tilde A_{ij}$ for each $i$ when $j \ge 2$, while $\tilde A_{1j} \cdot L^{(1)}_{i1}$ is independent of $\tilde A_{ij}$ for each $i \ge 2$, so that $B = A^{(2)}_{2:n,2:n}$ also satisfies the (I) hypothesis. Now we can apply the inductive hypothesis to $B$ to yield a GEPP factorization $PB=LU$ where $P \sim \Uniform(\mathcal P_{n-1})$. Embedding $\mathcal P_{n-1}$ into $\mathcal P_n$ using the map $Q \mapsto 1 \oplus Q$ yields then the resulting GEPP permutation matrix factor 
\begin{equation}
P_\sigma = (1 \oplus P)P^{(1)}
\end{equation}
for $A$. Moreover, since $P \sim \Uniform(\mathcal P_{n-1})$, then $1 \oplus P = P_\rho$ for $\rho \sim \Uniform(S_{n-1})$\footnote{Now associating $S_{n-1}$ with the isomorphic subgroup of $S_n$ that fixes 1.} while $P^{(1)} = P_{\sigma_1}$ for $\sigma_1 \sim \Uniform\{(1 \ j): j=1,2,\ldots,n\}$, so that by the Subgroup algorithm we have $\sigma = \rho \sigma_1 \sim \Uniform(S_n)$. It follows $P_\sigma \sim \Uniform(\mathcal P_n)$. This establishes the first statement part of (I) from \Cref{thm:main}.

The prior conversation up to \eqref{eq: num pivots} establishes the correspondence
\begin{align*}
    \P({\Pi(A)} = n-k) &= \frac{\#\{\sigma \in S_n: j = i_j \mbox{ for $k$ indices $j$}\}}{n!}\\
    &= \frac{\#\{\sigma \in S_n: \mbox{$\sigma$ has $k$ cycles in its disjoint cycle decomposition}\}}{n!} \\
    &= \frac{|s(n,k)|}{n!}\\
    &=\P(\Upsilon_n = k)
\end{align*}
for $k = 1,2,\ldots,n$, which  yields the desired result {$\Pi(A) \sim n-\Upsilon_n$}.
\end{proof}

\section{Butterfly permutations}
\label{sec:butterfly}

Using the Kronecker product factorization for $B(\boldsymbol\t) \in \B_s(N)$ (where $N = 2^n$) along with the mixed-product property, then matrix factorizations of each Kronecker component yield the matrix factorization of the resulting butterfly matrix. This was used in \cite{jpm} to yield both the eigenvalue (Schur) decomposition as well as the $LU$ factorizations of scalar simple butterfly matrices, $\B_s(N)$. For {reference}, we will restate this latter result: 
\begin{proposition}[\cite{jpm}]
    \label{prop: ge factors}
    Let $B = B(\boldsymbol \t) \in \B_s(N)$. 
    
    \normalfont{(I)} If $\cos\t_i \ne 0$ for all $i$, then $B$ has the GENP factorization $B = L_{\boldsymbol \t}U_{\boldsymbol \t}$ where $L_{\boldsymbol\t} = \bigotimes_{j=1}^n L_{\t_{n-j+1}}$ and $U_{\boldsymbol \t} = \bigotimes_{j=1}^n U_{\t_{n-j+1}}$ for
    \begin{equation}
    \label{eq: genp factors}
        L_{\t} = \begin{bmatrix}
            1 &0 \\-\tan\t & 1
        \end{bmatrix} \quad \mbox{and} \quad U_{ \t} =  \begin{bmatrix}
            \cos\t & \sin\t\\
            0&\sec\t
        \end{bmatrix}.
    \end{equation}
    
    \normalfont{(II)} Using $\boldsymbol \t \in [0,2\pi)^n$, let 
    \begin{equation}
    {e_j = \left\{\begin{array}{cl}
    1 & \mbox{if $|\tan \theta_j| > 1$,}\\
    0 & \mbox{if $|\tan \theta_j| \le 1$}
    \end{array}\right.}
    \end{equation}
    for each $j$. Let $\boldsymbol \t ' \in [0,2\pi)^n$ be such that $\theta_j' = \frac\pi2e_j + (-1)^{e_j} \t_j = \theta_j e_j + (\frac\pi2 - \theta_j)(1-e_j)$ for each $j$. If $|\tan \t_j| \ne 1$ for any $j$, then the GEPP  factorization of $B$ is $PB = LU$ where  $P = P_{\boldsymbol \t} = \bigotimes_{j=1}^n P_{\t_{n-j+1}}$, $L = L_{\boldsymbol\t'}$, and $U = U_{\boldsymbol \t'} D_{\boldsymbol \t}$ for 
    \begin{equation}
    P_\t = P_{(1 \ 2)}^{e_j} \quad \mbox{and} \quad D_\t = (-1)^{1-e_j} \oplus 1.
    \end{equation}
    Moreover,  $(PB)^{(k)} = B(\boldsymbol \t')^{(k)}D_{\boldsymbol \t}$ for all $k$ where $B(\boldsymbol\t') \in \B_s(N)$.
\end{proposition}

In particular, $P \in \mathcal P_N^{(B)}$ (cf. \eqref{eq:def butterfly perm}) for $PB = LU$ the GEPP factorization of $B \in \B_s(N)$. Note if $\theta \sim \Uniform([0,2\pi))$ then $\mathbb P(|\tan \theta| \le 1) = \frac12$, so the resulting GEPP permutation matrix factor $P_{\boldsymbol \t}$ for $B(\boldsymbol \t) \sim \B_s(N,\Sigma_S)$, a Haar-butterfly matrix, then satisfies
\begin{equation}
\P(P_{\boldsymbol \t} = Q) = 2^{-n} = \frac1N = \frac1{|\mathcal P_N^{(B)}|}
\end{equation}
for each $Q \in \mathcal P_N^{(B)}$ (using also  \Cref{prop:haar butterfly,prop: ge factors}). This establishes the first part of (II) from \Cref{thm:main}:
\begin{corollary}
\label{cor: haar perm butterfly}
If $B \sim \B_s(N,\Sigma_S)$, then $P \sim \Uniform(\mathcal P_N^{(B)})$ for the GEPP  factorization $PB = LU$.
\end{corollary}



Next, we can connect the resulting GEPP permutation matrix factors for Haar-butterfly matrices to the corresponding number of total pivot movements needed. This can be accomplished by finding the associated permutation $\sigma \in S_n$ written in the form \eqref{eq:perm_form} for the corresponding butterfly permutation. 
\begin{proposition}
\label{prop: butterfly perm form}
If $n = 1$, $\mathcal P_2^{(B)} = \mathcal P_2$. For $n > 1$, then $P_\sigma \in \mathcal P_{2N}^{(B)}$ where $\sigma \in S_{2N}$ written in the form \eqref{eq:perm_form} is one of four options: either $\sigma = 1$ if $P_\sigma = {\V I_{2N} = \V I_2 \otimes \V I_N}$ or $\sigma \ne 1$ is the product of $N$ disjoint transpositions, where  $\sigma = (N \ 2N)\cdots(2 \ N+2)(1 \ N+1)$ if $P_\sigma = P_{(1 \ 2)} \otimes \V I_N$, or
\begin{equation}
\begin{array}{ll}
\sigma = (a_1+N \ a_2 + N) \cdots (a_{N-1} + N \ a_N+N)(a_1 \ a_2) \cdots (a_{N-1} \ a_N) &\mbox{if $P_\sigma = \V I_2 \otimes P_\rho$, or}\\
\sigma = (a_2 \ a_1+N)(a_1 \ a_2+N) \cdots (a_N \ a_{N-1} + N)(a_{N-1} \ a_N+N) &\mbox{if $P_\sigma = P_{(1\ 2)} \otimes P_\rho$,}
\end{array}
\end{equation}
where $P_\rho \in \mathcal P_N^{(B)}$ with $\rho = (a_1 \ a_2) \cdots (a_{N-1} \ a_N) \in S_N$ in the form \eqref{eq:perm_form} such that $a_{2k-1} < a_{2k}$ for each $k$ unless $\rho = 1$ and $a_{2k-1} > a_{2k+1}$ for each $k$ when $n > 2$. 

\end{proposition}
\begin{proof}
We will use induction on $n = \log_2 N$ along with the fact $\mathcal P_{2N}^{(B)} = \mathcal P_2^{(B)} \otimes \mathcal P_N^{(B)}$. For $n = 2$, starting with $\mathcal P_2^{(B)} = \mathcal P_2 = \{\V I_2, P_{(1 \ 2)}\}$ we have
\begin{align}
\mathcal P_4^{(B)} &= \mathcal P_2 \otimes \mathcal P_2 = \{\V I_2 \otimes \V I_2, P_{(1 \ 2)} \otimes I_2, \V I_2 \otimes P_{(1 \ 2)}, P_{(1 \ 2)} \otimes P_{(1 \ 2)}\}\\
&= \{\V I_4, P_{(2 \ 4)(1 \ 3)}, P_{(3 \ 4)(1 \ 2)},P_{(2\ 3)(1 \ 4)}\}.
\end{align}
The corresponding permutations as written above then all satisfy the form \eqref{eq:perm_form}, where we abstain from including the trivial permutations $(i_k \ k)$ when $i_k = k$. Moreover, writing each associated non-trivial permutation can be written as the product of $N=2$ disjoint transpositions of the form $(a_1 \ a_2)(a_3 \ a_4)$, which then have $a_1 > a_3$ and $a_{2k-1} < a_{2k}$ for $k = 1,2$. (Note the (distinct) transpositions used must necessarily be disjoint since necessarily $P_\sigma^2 = P_{\sigma^2} =  \V I$.)

Assume the result holds for $n$. The $n+1$ case follows by just reading off the corresponding permutations $\V I_2 \otimes P_\rho$ and $P_{(1\ 2)} \otimes P_\rho$ for $\rho = (a_1 \ a_2) \cdots (a_{N-1} \ a_N)$ such that $P_\rho \in \mathcal P_N^{(B)}$ (for $\rho$ already in form \eqref{eq:perm_form}). For $\rho = 1$, then $P_1 = \V I_N$ yields $\V I_2 \otimes \V I_N = \V I_{2N}$ and $P_{(1\ 2)} \otimes \V I_N = P_{(N \ 2N)\cdots(2 \ N+2)(1 \ N+1)}$; if $P_\rho \in \mathcal P_N^{(B)}$ for $\rho = (a_1 \ a_2)\cdots (a_{N-1} \ a_N) \in S_N$ where $a_{2k+1} < a_{2k-1}<a_{2k}$ for each $k$, then
\begin{equation}
\label{eq:above}
\begin{array}{l}
P_\sigma = \V I_2 \otimes P_\rho = P_{(a_1+N \ a_2 + N) \cdots (a_{N-1} + N \ a_N+N)(a_1 \ a_2) \cdots (a_{N-1} \ a_N)} \quad \mbox{and}\\
P_\sigma = P_{(1\ 2)} \otimes P_\rho = P_{(a_2 \ a_1+N)(a_1 \ a_2+N) \cdots (a_N \ a_{N-1} + N)(a_{N-1} \ a_N+N)}.
\end{array}
\end{equation}
Moreover, each associated permutation $\sigma$ in the form \eqref{eq:perm_form} then is either the trivial permutation or is the product of $N$ disjoint transpositions and can be written in the form $(b_1 \ b_2) \cdots (b_{2N-1} \ b_{2N})$ where $b_{2k+1} < b_{2k-1} < b_{2k}$ for each $k$. This holds directly by construction for the associated permutations for $\V I_2 \otimes \V I_N$, $P_{(1\ 2)} \otimes \V I_N$ and $\V I_2 \otimes P_\rho$, which follows since $a_k \le N$ along with $a_{2k+1}<a_{2k-1} < a_{2k}$ for all $k$. It remains to show $\sigma$ can be written in this form when $P_\sigma = P_{(1 \ 2)} \otimes P_\rho$. 

By \eqref{eq:above}, $\sigma = (a_2 \ a_1 + N)(a_1 \ a_2 + N) \cdots (a_N \ a_{N-1} + N)(a_{N-1} \ a_N + N)$. Note this is the product of $N$ disjoint transpositions again since $a_k \le N$ and $a_i \ne a_j$ for all $i \ne j$. Moreover, since the $a_j$ are distinct, we can label $b_{2k-1} = a_{(N-k+1)}$ for $k=1,2,\ldots,N$ using the ordered statistics subscript (so $a_{(1)} < a_{(2)} < \cdots < a_{(N)}$), with then $b_{2k} = \rho(a_{(N-k+1)}) + N$. We can then further write  $\sigma = (b_1 \ b_2) \cdots (b_{2N-1} \ b_{2N})$ since disjoint cycles commute, for which it then follows also $b_{2k+1} < b_{2k-1} < b_{2k}$ for each $k$.
\end{proof}

\begin{example}
\label{ex:butterfly perms}
The corresponding permutations are $\{1,(1 \ 2)\} = S_2$ for $\mathcal P_2^{(B)}$,  $\{$$1$, $(2 \ 4)(1 \ 3)$, $(3 \ 4)(1\ 2)$, $(2 \ 3)(1 \ 4)\}$ for $\mathcal P_4^{(B)}$, and
\begin{equation}
\label{eq:P_8}
    \left\{\begin{array}{l}1,(4 \ 8)(3 \ 7)(2 \ 6)(1 \ 5), \\
    (6 \ 8)(5 \ 7)(2 \ 4)(1 \ 3), (4 \ 6)(3 \ 5)(2\ 8)(1 \ 7), \\
    (7\ 8)(5\ 6)(3\ 4)(1 \ 2), (4 \ 7)(3 \ 8)(2 \ 5)(1\ 6), \\
    (6 \ 7)(5 \ 8)(2 \ 3)(1 \ 4), (4 \ 5)(3 \ 6)(2 \ 7)(1 \ 8) \end{array} \right\}
\end{equation}
for $\mathcal P_8^{(B)}$.
\end{example}

\begin{remark}
If no pivot movements are needed on the first GE step, then no pivot movements will be needed at any GE step (using \eqref{eq:perm_form} and \Cref{prop: butterfly perm form}). Furthermore, it follows inductively that half the time all of the pivot movements occur in the first $N/2$ steps; a quarter of the time all of the pivot movements occur as the first $N/4$ steps of each $N/2$ partitioning; an eighth of the time all of the pivots occur at the first $N/8$ steps of each $N/4$ partitioning; and so on, where $2^{-k}$ of the time the pivots occur at each step in the first half of each $N2^{1-k}$ partitioning of the GE steps, which stops with exactly one permutation (i.e., $2^{-n}=\frac1N$ of the time) does pivoting occur at precisely every other GE step. The remaining permutation accounts for no pivoting ever occurring when $P_{\boldsymbol \t} = \V I$. This yields a Cantor set-like decomposition of $[n]$ into the possible configuration of locations of where the pivots can occur using GEPP. To find out which configuration will be used on a particular simple scalar butterfly matrix (i.e., the exact pivoting locations one will encounter), this is determined by the first step where $k = i_k$. 

For example, using the associated permutations from $\mathcal P_8^{(B)}$ from \Cref{ex:butterfly perms}, we see 4 permutations have pivot movements only at the first half of the total GE steps (i.e., $(4\ 8)(3 \ 7)(2 \ 6),(1 \ 5)$, $(4 \ 6)(3 \ 5)(2 \ 8)(1 \ 7)$, $(4 \ 7)(3 \ 8)(2 \ 5)(1 \ 6)$, and $(4 \ 5)(3 \ 6)(2 \ 7)(1\ 8)$ each yield pivot movements at GE steps 1 through 4); 2 permutations have pivot movements only at the first half of each half partitioning of the total GE steps (i.e, $(6 \ 8)(5 \ 7)(2 \ 4)(1 \ 3)$ and $(6 \ 7)(5 \ 8)(2 \ 3)(1 \ 4)$ yield pivot movements only at steps 1,2 and 5,6); 1 partition has pivot movements at every other GE step (i.e., $(7 \ 8)(5 \ 6)(3 \ 4)(1 \ 2)$ yields pivot movements only at steps 1,3,5,7); with only one permutation (the trivial permutation) having no GE pivot movements.

\begin{figure}[htbp]
  \centering
  \includegraphics[width=0.8\textwidth]{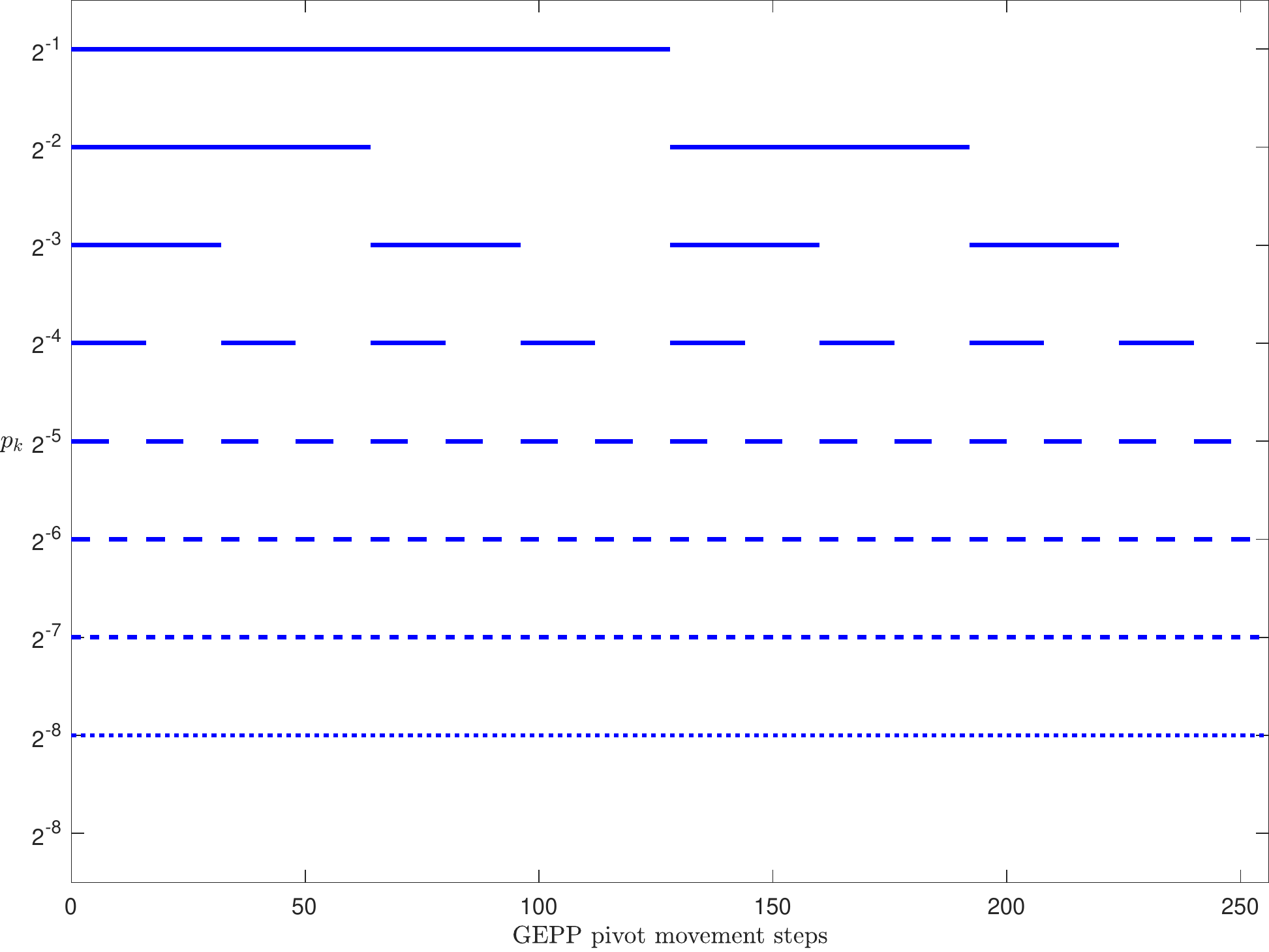}
  \caption{GEPP pivot movement configurations for Haar-butterfly permutations for $N=2^8$ and their associated probabilities, $p_k$, with the exact pivot movement locations indicated by blue}
  \label{fig:configs}
\end{figure}

 Moreover, applying GEPP to Haar-butterfly matrices, where then each butterfly permutation occurs with equal probability, then induces a probability measure on these configurations. \Cref{fig:configs} shows the possible GEPP pivot movement locations associated with Haar-butterfly permutations of size $N=2^8$, along with the probability for each particular configuration.

\end{remark}

\subsection{Proof of (II) of \Cref{thm:main}}
Now we have the sufficient tools to establish:
\begin{quotation}\normalsize
\noindent\textsc{\Cref{thm:main}}: (II) \emph{If $B \sim \B_s(N,\Sigma_S)$, then $P \sim \Uniform(\mathcal P_N^{(B)})$ for $PB = LU$ the GEPP factorization and ${\Pi(B)} \sim \frac{N}2 \Bernoulli\left(1-\frac1N\right)$.}
\end{quotation}
\begin{proof}
\Cref{cor: haar perm butterfly} yields directly that applying GEPP to a Haar-butterfly matrix results in a uniform butterfly permutation matrix factor, which is the first part of (II). Moreover, using the associated permutations from GEPP in the form \eqref{eq:perm_form} to then be able to read off explicitly which GE steps need pivot movements, then we have by \Cref{prop: butterfly perm form} that $i_k > k$ precisely for $N/2$ indices $k$ for each non-trivial case and for precisely 0 times in the trivial case. It follows then for $B = B(\boldsymbol\t) \sim \B_s(N,\Sigma_S)$, then since $P_{\boldsymbol \t} \sim \Uniform(\mathcal P_N^{(B)})$ we have
\begin{align}
\P({\Pi(B)} = 0) &= \P(P_{\boldsymbol \t} = \V I_N) = 2^{-n} = \frac1N, \mbox{ and}\\
\P\left({\Pi(B)} = \frac{N}2\right) &= \P(P_{\boldsymbol \t} \ne \V I_N) = 1 - \P(P_{\boldsymbol \t} = \V I_N) = 1-\frac1N.
\end{align}
Hence, ${\Pi(B)} \sim \frac{N}2 \Bernoulli(1 - \frac1N)$.
\end{proof}


\section{Random ensembles  \texorpdfstring{{\boldmath$\mathcal{PL}_n^{\max}(\xi)$}}{2}, \texorpdfstring{{\boldmath$\mathcal{PL}_n(\xi)$}}{2}, \texorpdfstring{{\boldmath$\mathcal{PL}_n^{\max}(\xi,\alpha)$}}{2} and \texorpdfstring{{\boldmath$\mathcal{PL}_n(\xi,\alpha)$}}{2}}
\label{sec:new}

Using the prior discussions, we can build a random ensemble of matrices that require a maximal number of GEPP pivot movements. As mentioned in \Cref{rmk:max pivot}, a maximal number of GEPP pivot movements occurs when the corresponding permutation matrix factor $P \in \mathcal P_n^{n\operatorname{-cycles}}$. Using this, we can define
\begin{equation}
\label{eq: max pivot random ensemble}
\mathcal{PL}_n^{\max}(\xi) = \{PL: P \sim \Uniform(\mathcal P_n^{n\operatorname{-cycles}}), L \in \mathcal L_n \mbox{ independent of $P$}, L_{ij} \sim \xi \mbox{ iid for  $i>j$}\}.
\end{equation}
Recall, by construction, the corresponding GEPP lower triangular matrix factor $L$ satisfies  $|L_{ij}| \le 1$ for all $i > j$. Moreover, \Cref{thm:GEPP_unique} yields the $L$ factor is invariant under row permutations if $|L_{ij}| < 1$ for all $i > j$. Hence, if $A = PLU$ where $U \in \mathcal U_n$ and {$PL \sim \mathcal{PL}_n^{\max}(\xi)$} for any distribution $\xi$ with $|\xi| < 1$, then $A$ will always require $n-1$ GEPP pivot movements. Similar ensembles can be constructed where the $P$ factor is restricted to particular pivot configurations. 

We will also study the more general model
\begin{equation}
\label{eq: gen random ensemble}
\mathcal{PL}_n(\xi) = \{PL: P \sim \Uniform(\mathcal P_n), L \in \mathcal L_n \mbox{ independent of $P$}, L_{ij} \sim \xi \mbox{ iid for $i>j$}\}.
\end{equation}
When $|\xi| < 1$, then $A=PLU$ for $PL \sim \mathcal{PL}_n(\xi)$ and $U \in \mathcal U_n$ corresponds to the GEPP $LU$ factorization of $A$.

\begin{remark}
Two natural distributions to consider are $\xi \sim \Uniform([-1,1])$ and $\xi \sim \Uniform(\mathbb D)$. Using GEPP on $\GL_n(\mathbb F)$, the left-coset representatives of $\mathcal U_n(\mathbb F)$ in $\GL_n(\mathbb F)$, $\GL_n(\mathbb F)/\mathcal U_n(\mathbb F)$, are precisely of the form $PL$ for $P \in \mathcal P_n$ and $L \in \mathcal L_n$ where $|L_{ij}|\le 1$ for all $i > j$. Hence, $\mathcal PL_n(\xi)$ then corresponds to uniformly sampled representatives of $\GL_n(\mathbb R)/\mathcal U_n(\mathbb R)$ when $\xi \sim \Uniform([-1,1])$ and uniformly sampled representatives of $\GL_n(\mathbb C)/\mathcal U_n(\mathbb C)$ when $\xi \sim \Uniform(\mathbb D)$.
\end{remark}



\subsection{Eigenvalues of \texorpdfstring{{\boldmath$\mathcal{PL}_n^{\max}(\xi)$}}{2} and \texorpdfstring{{\boldmath$\mathcal{PL}_n(\xi)$}}{2}}
Suppose $PL \sim \mathcal{PL}_n^{\max}(\xi)$. Since $L \in \mathcal L_n$, then its eigenvalues are exactly 1 with multiplicity $n$. Since $P \in \mathcal P_n^{n\operatorname{-cycles}}$, then its eigenvalues are exactly the $n^{th}$ roots of unity, $e^{2\pi i/n}$. The spectral pictures for each $P$ and $L$ separately fall exactly on $\partial \mathbb D$, and these are deterministic despite each matrix being random. 

So a natural next question is what does the spectral picture look like for their product, $PL$? The eigenvalues no longer stay on $\partial \mathbb D$, but they appear to remain asymptotically concentrated inside $\mathbb D$ when scaled by $\sqrt{n \sigma^2/2}$ when $\mathbb E \xi = 0$ (i.e., $\xi$ is centered) and $\sigma^2 = \mathbb E |\xi|^2$ is the variance of $\xi$. \Cref{fig:PL eig} shows the (computed) eigenvalue locations for $\mathcal{PL}_n^{\max}(\xi)$ scaled by $\sqrt{n \sigma^2/2}$ using $n=2^{14}=16,384$ and $\xi$ sampled from $\Uniform([-1,1])$, $\Uniform(\mathbb D)$, Rademacher and $N(0,1)$. Noticeably, \Cref{fig:PL eig} further suggests a universality result for this limiting behavior.

Recall the empirical spectral distribution (ESD) of a $n\times n$ matrix $A$ is the probability measure
\begin{equation}
\label{eq:esd}
\mu_A = \frac1n \sum_{k=1}^n \delta_{\lambda_k(A)},
\end{equation}
which gives equal weight to the location of each eigenvalue of $A$. Note if $A$ is a random matrix, then $\mu_A$ is a random probability measure. 

Empirically, \Cref{fig:PL eig} then suggests $\mu_{A}$ is converging to a probability measure on $\mathbb D$ that is an interpolation between the uniform measure on $\mathbb D$ and the Dirac measure at the origin when $A = PL/\sqrt{n \sigma^2/2}$ for $PL \sim \mathcal{PL}_n^{{\max}}(\xi)$ with $\xi$ having 0 mean and finite variance. Although the eigenvalues of $A$ have a higher density near the origin, they can never include the origin since $PL$ is nonsingular.

\begin{figure}[htp] 
    \centering
    \subfloat[$\xi \sim \Uniform(\mbox{[$-1,1$]})$]{%
        \includegraphics[width=0.52\textwidth]{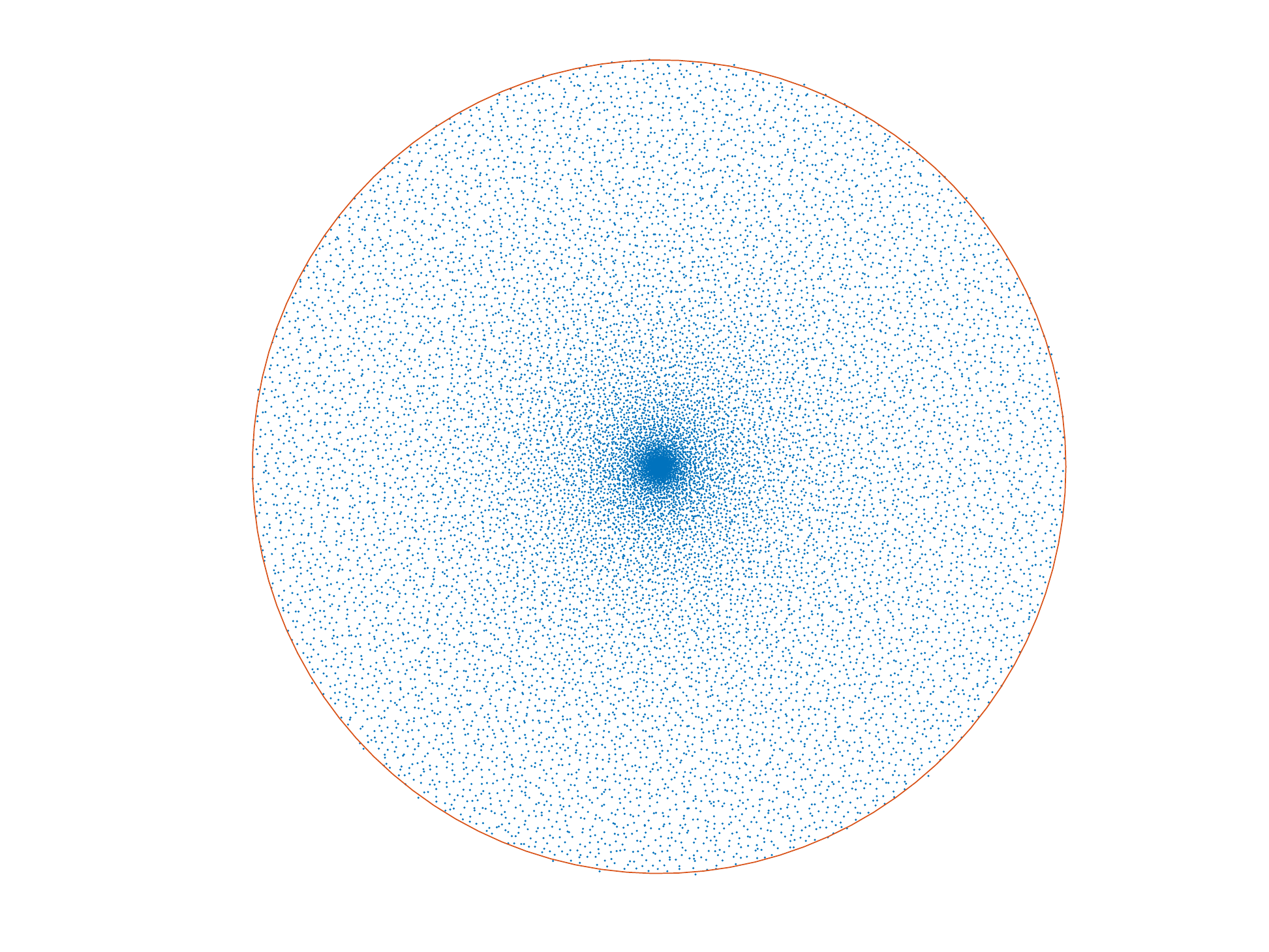}%
        \label{fig:PL real}%
        }%
    \hspace{-2pc}%
    \subfloat[$\xi \sim \Uniform(\mathbb D)$]{%
        \includegraphics[width=0.52\textwidth]{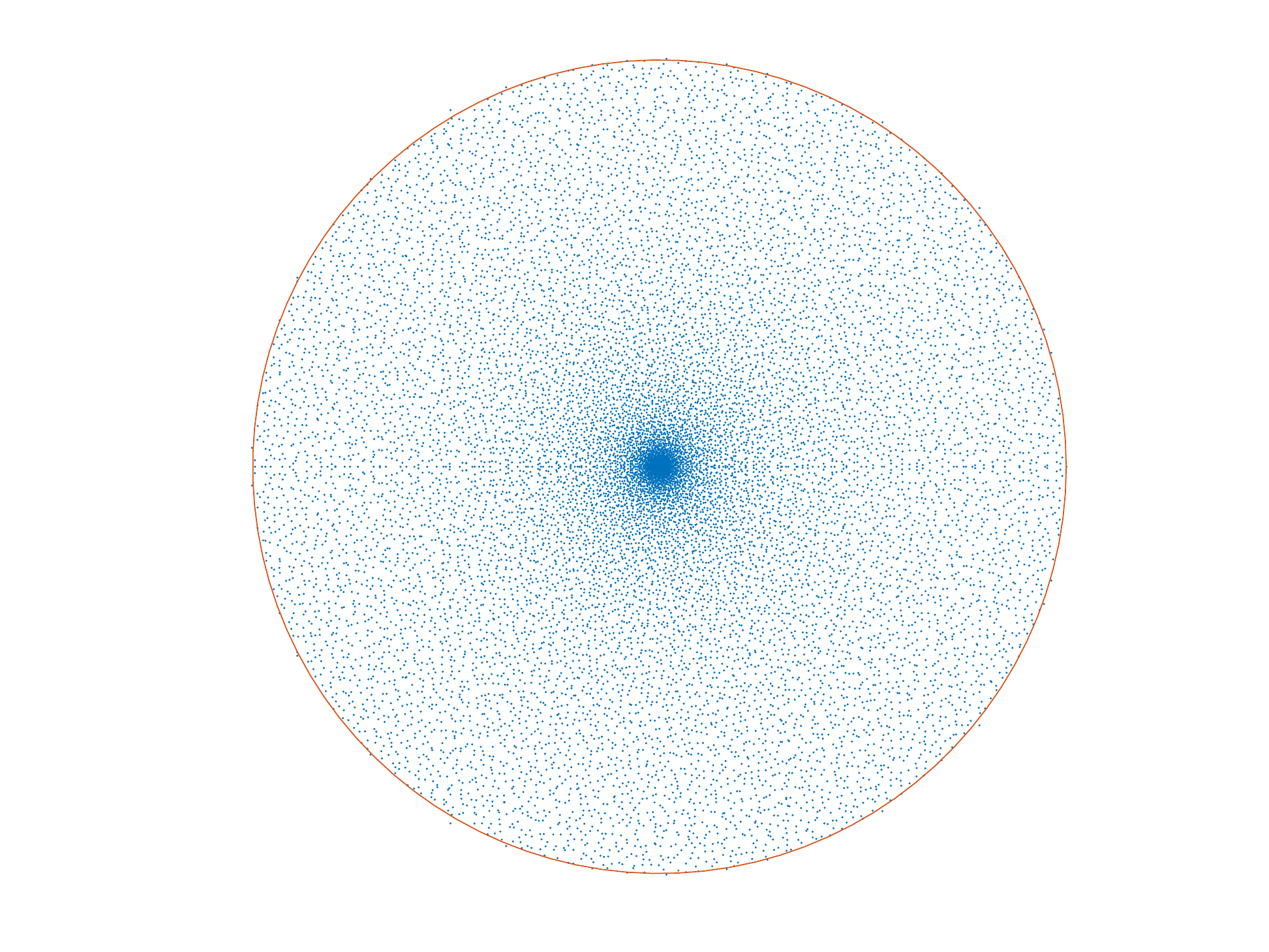}%
        \label{fig:PL complex}%
        }%
    \\
    \subfloat[$\xi \sim \operatorname{Rademacher}$]{%
        \includegraphics[width=0.52\textwidth]{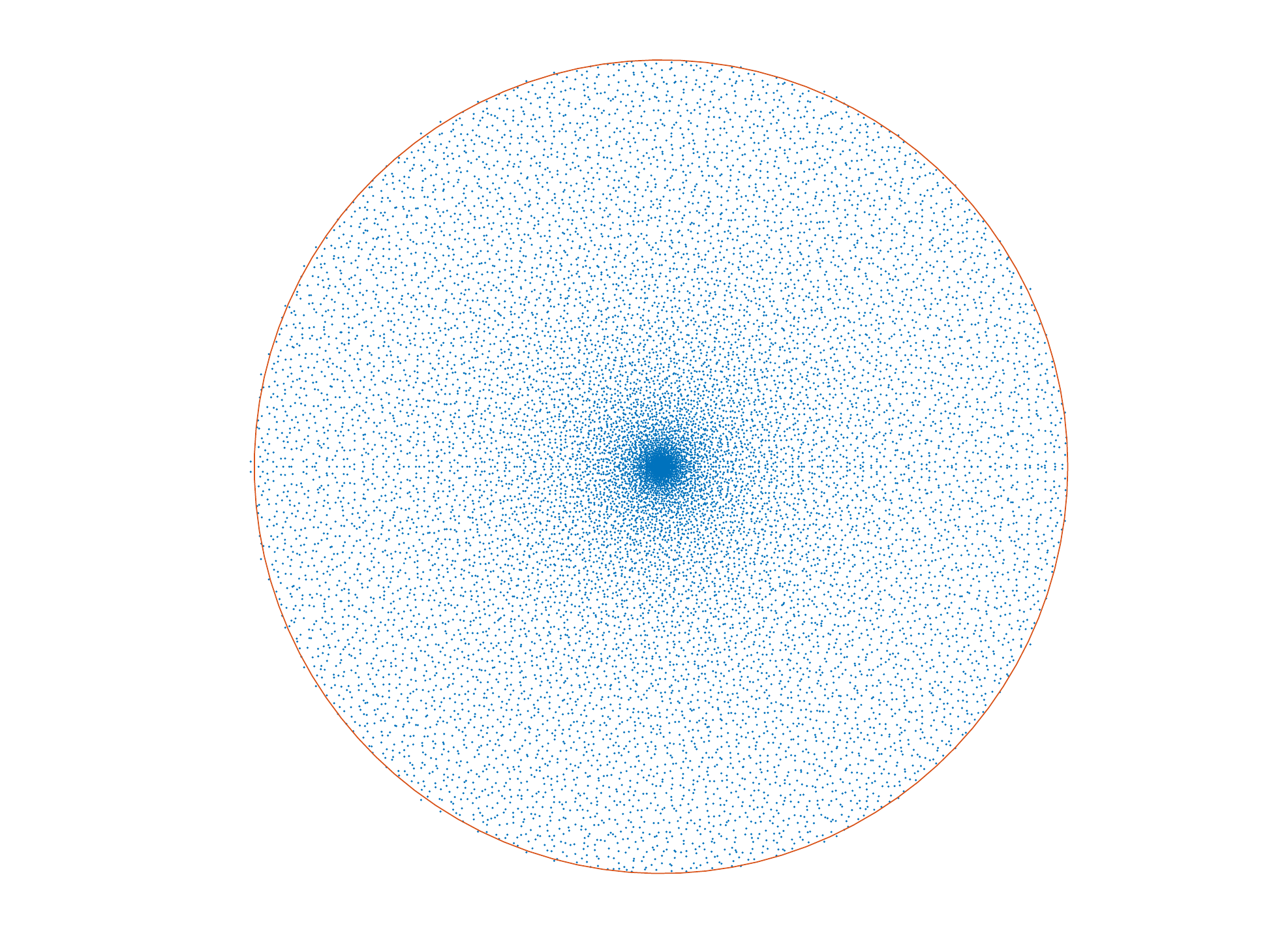}%
        \label{fig:PL Rademacher}%
        }%
    \hspace{-2pc}%
    \subfloat[$\xi \sim N(0,1)$]{%
        \includegraphics[width=0.52\textwidth]{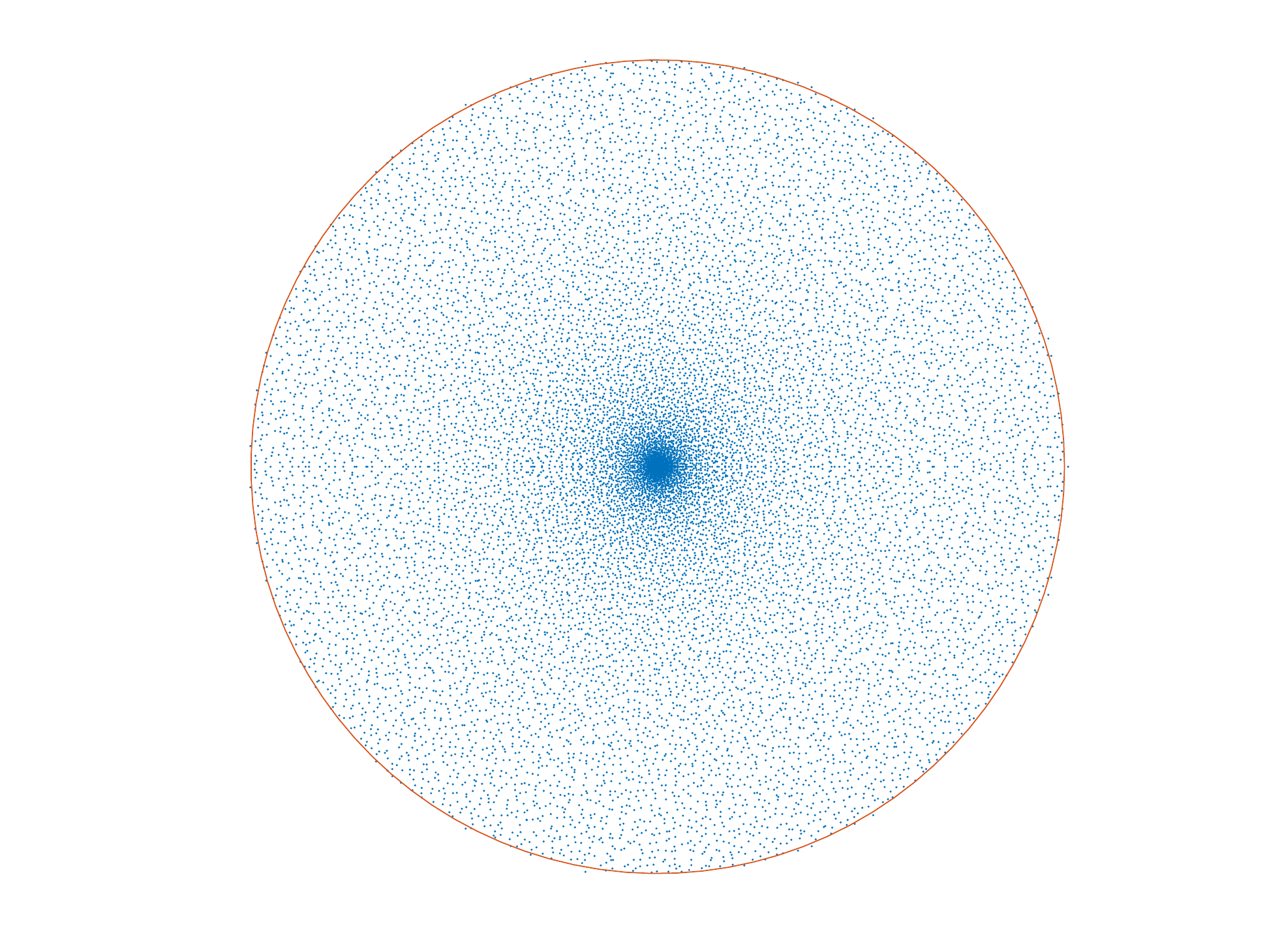}%
        \label{fig:PL norm}%
        }%
    \caption{Computed eigenvalues (in blue) for $PL/\sqrt{n \sigma^2/2}$ where $n = 2^{14}=16,384$ and $PL \sim {\mathcal{PL}}_n^{\max}(\xi)$ for (a) $\xi \sim \Uniform([-1,1])$ where $\sigma^2 = \frac13$, (b) $\xi \sim \Uniform(\mathbb D)$ where $\sigma^2 = \frac12$, (c) $\xi \sim \operatorname{Rademacher}$ where $\sigma^2 = 1$, and (d) $\xi \sim N(0,1)$ where $\sigma^2 = 1$, mapped against the unit complex circle $\partial \mathbb D$ (in red)}
    \label{fig:PL eig}
    \end{figure}


The pictures are indistinguishable to \Cref{fig:PL eig} when instead using $\mathcal{PL}_n(\xi)$, where $P \sim \Uniform(\mathcal P_n)$ instead of $P \sim \Uniform(\mathcal P_n^{n\operatorname{-cycles}})$. In both cases, the ESD of $P$ limit to the uniform measure on {$\partial \mathbb D$} in probability in the weak star sense \cite{DiSh94}. However, replacing $P$ with another random matrix from an ensemble whose ESDs similarly limit to the uniform measure on {$\partial  \mathbb D$} (e.g., using Haar sampling for $\O(n)$, $\U(n)$, or $\B_s(n)$ \cite{DiSh94,Tr19}) yield different spectral pictures that extend beyond $\mathbb D$ when still using the scaling $\sqrt{n \sigma^2/2}$. This points to the significance of the divisor of $2$ in the above scaling, which corresponds to the {sparsity (or density)} of the $L$ factor of $\frac12$; for example, $UL$ has full density a.s. when $U \sim \Haar(\O(n))$ or $U \sim \B_s(N,\Sigma_S)$. 

Since multiplication by permutation matrices preserves {sparsity}, one might expect the same scaling when uniformly sampling permutation matrices versus corresponding $n$-cycle permutation matrices. Both should adequately mix up the rows of $L$ so that the eigenvalues move away from 1. However, preserving {sparsity} is not sufficient on its own, as can be seen if using a diagonal matrix $D$, since then $DL$ will only have eigenvalues located at the diagonal entries of $D$ (since $L$ is  unipotent lower triangular). A future area of research can study a more general random ensemble that could replace $P$ in $PL \sim \mathcal{PL}_n(\xi)$ with another random matrix that  sufficiently randomizes the rows of $L$ without changing the sparsity of $L$.

\subsection{Fixed sparsity random ensembles \texorpdfstring{{\boldmath$\mathcal{PL}_n^{\max}(\xi,\alpha)$}}{2} and \texorpdfstring{{\boldmath$\mathcal{PL}_n(\xi,\alpha)$}}{2}}

We can similarly study random ensembles with (approximately) fixed sparsity. Let the sparsity of a matrix $A$ be determined by $\alpha=\alpha(A)$, the ratio of the number of zero entries of $A$ to the total number of entries of $A$. If $\alpha = 0$, then $A$ has full density, and if $\alpha = 1$ then $A$ is the zero matrix. A full lower triangular matrix has {sparsity} given by {$\alpha = 1-\binom{n}2/n^2=\frac12(1+\frac1n) \approx \frac12$} for large $n$. We can then construct a matrix with fixed (approximate) sparsity by zeroing out all entries above a  specific diagonal. For a $n \times n$ matrix that has zero entries only above a set diagonal $k \in [-n,n]$ with entries at indices $(i,i+k)$ (where $k=0$ is the main diagonal, $k>0$ is the super diagonal, and $k<0$ are the sub diagonals), one can determine the density by computing:
\begin{equation}
   { g_n(k) = \left\{\begin{array}{cl}
    \displaystyle\frac{(n+k)(n+k+1)}{2n^2}, & k < 0,\\
    \vspace{-.5pc}\\
    \displaystyle\frac12\left(1 - \frac1n\right), & k = 0,\\
    \vspace{-.5pc}\\
    \displaystyle1 - \frac{(n-k)(n-k+1)}{2n^2}, & k > 0.
    \end{array}\right.}
\end{equation}
This is the ratio of nonzero entries to the total number of matrix entries for a matrix that is zero above and full at or below the $k^{th}$ diagonal; the triangular numbers naturally show up for the numerator in the {each} term. {Note $g_n(k) < \frac12$ iff $k < 0$ and $g_n(k) > \frac12$ iff $k > 0$. Moreover,} $g_n(k)$ is a quadratic polynomial in $k$ for fixed $n$ {when $k < 0$ or $k > 0$}, so $g_n(k)$ can be extended to include non-integral $k$. In particular, one could uniquely solve for 
\begin{equation}
k_\alpha \in [-n,n] \quad \mbox{such that} \quad g_n(k_\alpha)={1-\alpha}.\footnote{{If $\alpha > \frac12$, $k_\alpha = -n + \frac12[-1 + \sqrt{1 + 8n^2(1-\alpha)}] \in \left[-n,-\frac12\right)$; if $\alpha = \frac12$, set $k_{1/2} = 0$; if $\alpha < \frac12$, $k_\alpha = n + \frac12[1 - \sqrt{1 + 8n^2 \alpha}] \in \left(\frac12,n\right]$.}}
\end{equation}
Using this, we can introduce another random ensemble
\begin{equation}
\label{eq: PL xi alpha}
\mathcal{PL}_n(\xi,\alpha) = \left\{PL: P \sim \Uniform(\mathcal P_n), L \mbox{ independent of $P$}, 
\begin{array}{ll}
L_{ij} = 0 &\mbox{if $i + \lfloor k_\alpha \rfloor \le j$,}\\ 
L_{ij} \sim \xi \mbox{ iid} & \mbox{if $i+\lfloor k_\alpha \rfloor>j$}
\end{array} 
\right\}
\end{equation}
We can similarly define $\mathcal{PL}_n^{\max}(\xi,\alpha)$ by requiring $P \sim \Uniform(\mathcal P_n^{n\operatorname{-cycles}})$ (as well as other ensembles with fixed pivot configurations). Note if $PL \sim \mathcal{PL}_n(\xi,\frac12)$ then $P(L+\V I_n) \sim \mathcal{PL}_n(\xi)$.

Known results for asymptotic limiting behavior of ESDs to probability measures on $\mathbb D$ include the famous Circular law introduced by Bai in \cite{Bai} and proven with a fourth moment condition by Tao and Vu \cite{TaoVu}.

\begin{theorem}[Circular law \cite{TaoVu}]
\label{thm:circular}
Let $A$ be a $n \times n$ matrix with iid entries sampled from $\xi$ where $\E \xi = 0$, $\sigma^2 = \E |\xi|^2$ and $\E |\xi|^4 < \infty$. Then $\mu_{A/\sqrt{n \sigma^2}}$ converges weakly to the uniform measure on $\mathbb D$ in probability and almost surely.
\end{theorem}

\Cref{thm:circular} yields the asymptotic spectral picture for $\mathcal{PL}_n(\xi,{0})$ {(i.e., with sparsity 0)} when $\xi$ is centered with finite fourth moment. \Cref{fig:Circular} shows the map of the computed eigenvalues $\mathcal{PL}_n(N(0,1),{0}) \sim \operatorname{Gin}(n,n)$ for $n = 2^{14}$, while \Cref{fig:n Uniform D points} plots $n$ random sample points from $\Uniform(\mathbb D)$. Comparing this to the Ginibre picture (i.e., \Cref{fig:Circular}) also exemplifies how the asymptotic result does not hold for fixed $n$. The Ginibre picture has repulsions between its eigenvalues that lead to points being similarly spaced, while the asymptotic endpoint (i.e., \Cref{fig:n Uniform D points}) has Poisson clumping.

\begin{figure}[htp] 
    \centering
    \subfloat[$\xi \sim N(0,1), {\alpha = 0}$]{%
        \includegraphics[width=0.52\textwidth]{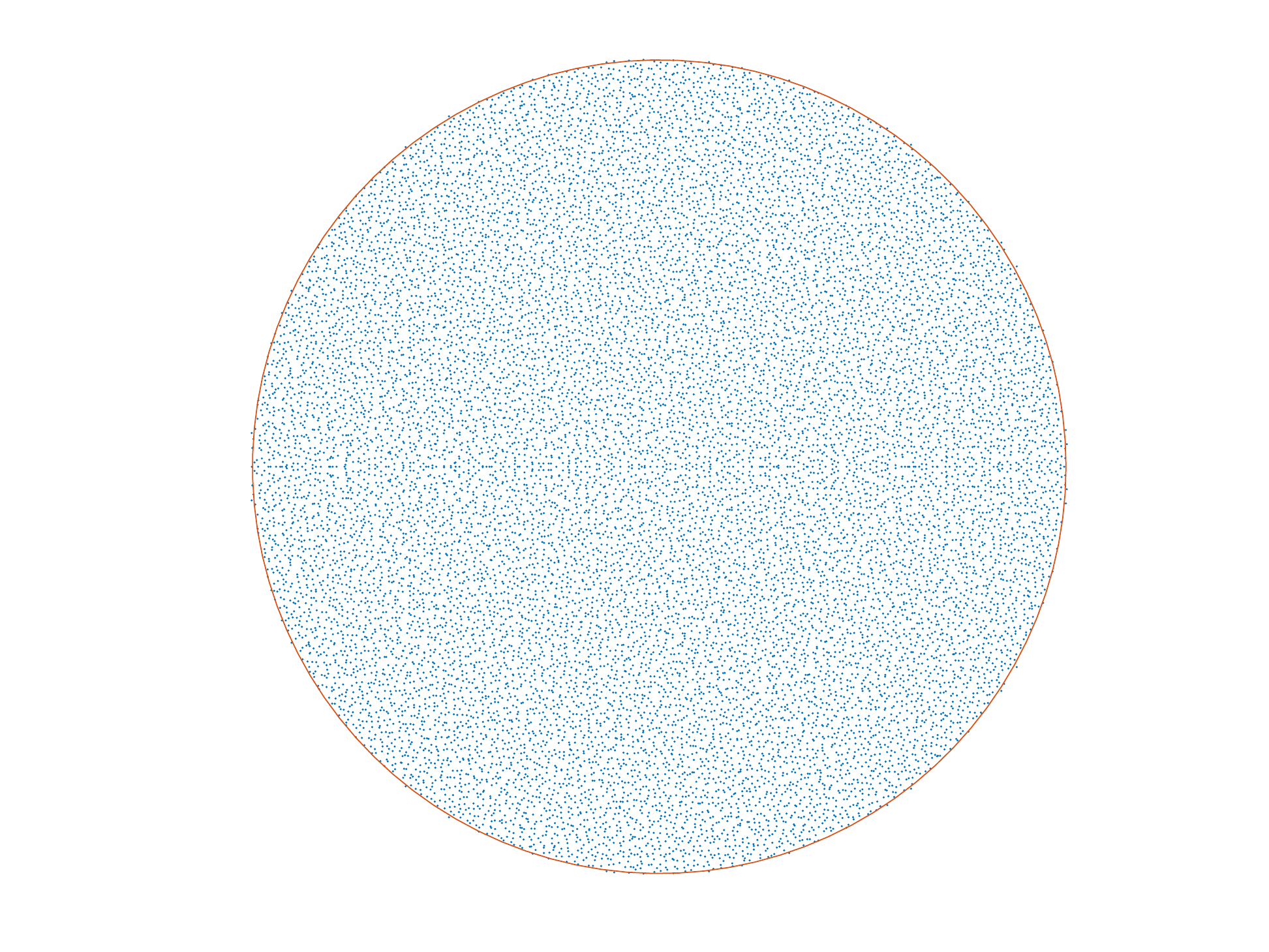}%
        \label{fig:Circular}%
        }%
    \hspace{-2pc}%
    \subfloat[$n=2^{14}$ iid samples from $\Uniform(\mathbb D)$]{%
        \includegraphics[width=0.52\textwidth]{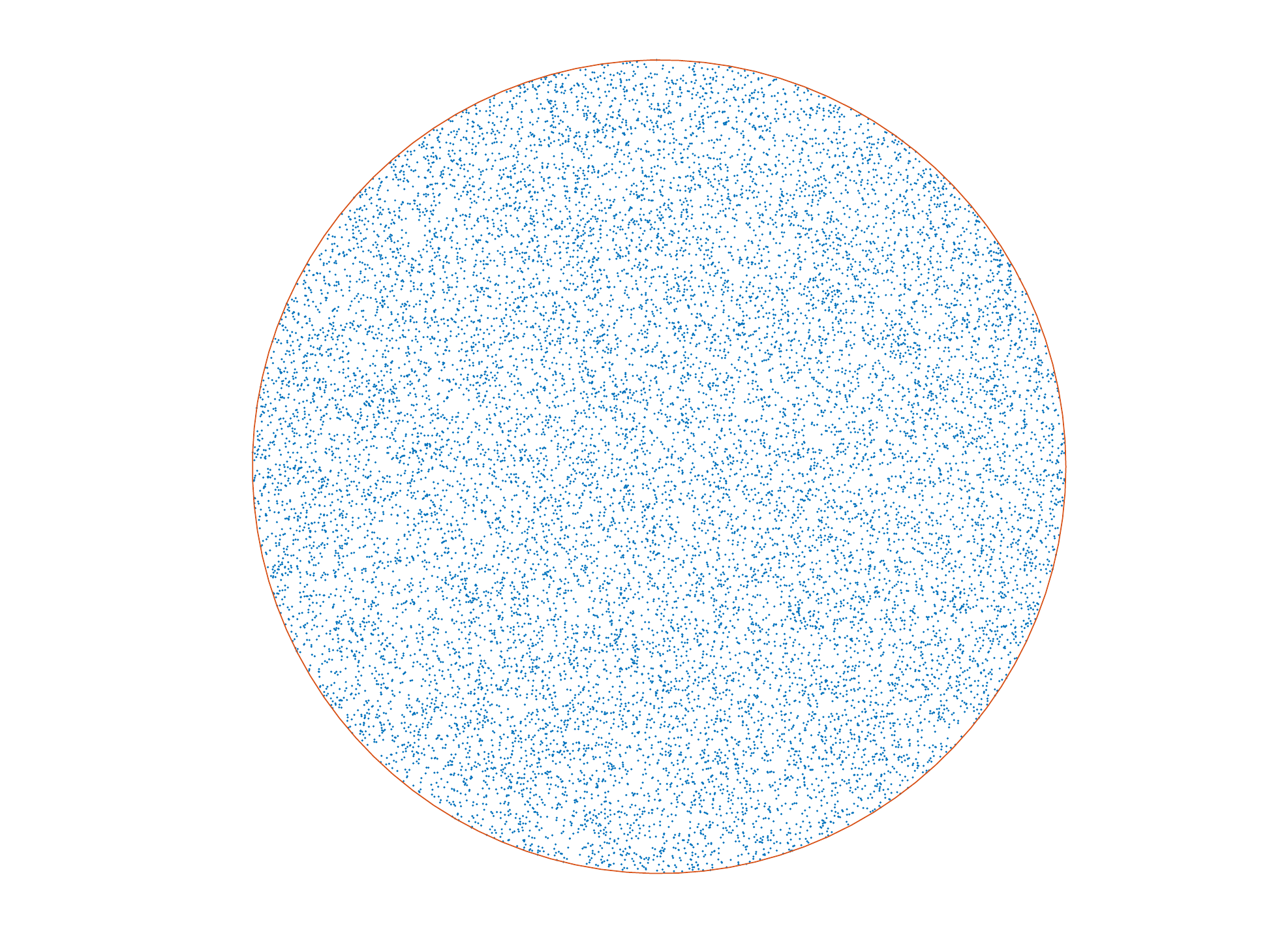}%
        \label{fig:n Uniform D points}%
        }%
    \\
    \subfloat[$\xi \sim N(0,1), {\alpha = \frac14}$]{%
        \includegraphics[width=0.52\textwidth]{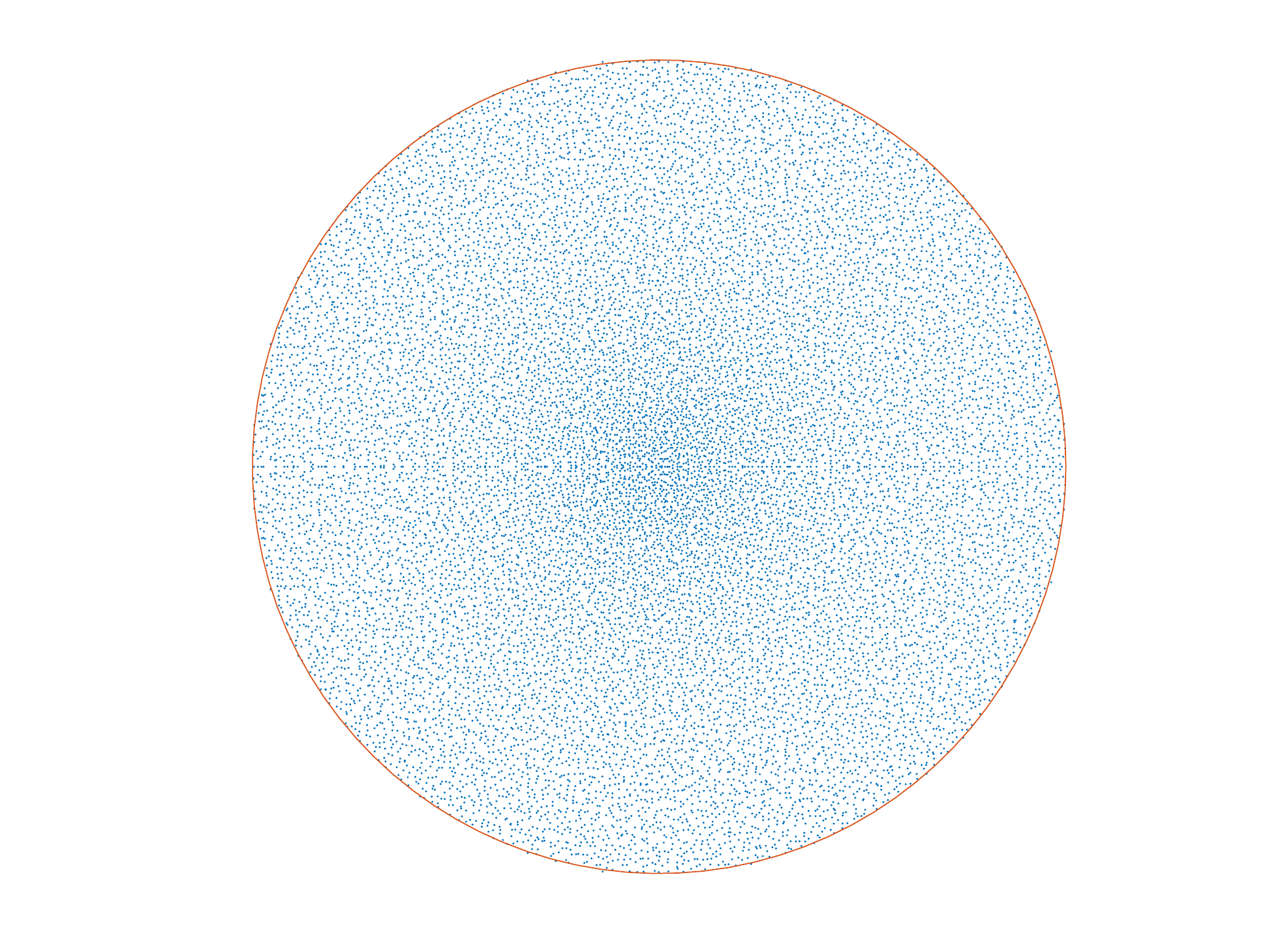}%
        \label{fig:a=3/4}%
        }%
    \hspace{-2pc}%
    \subfloat[$\xi \sim N(0,1), {\alpha = \frac34}$]{%
        \includegraphics[width=0.52\textwidth]{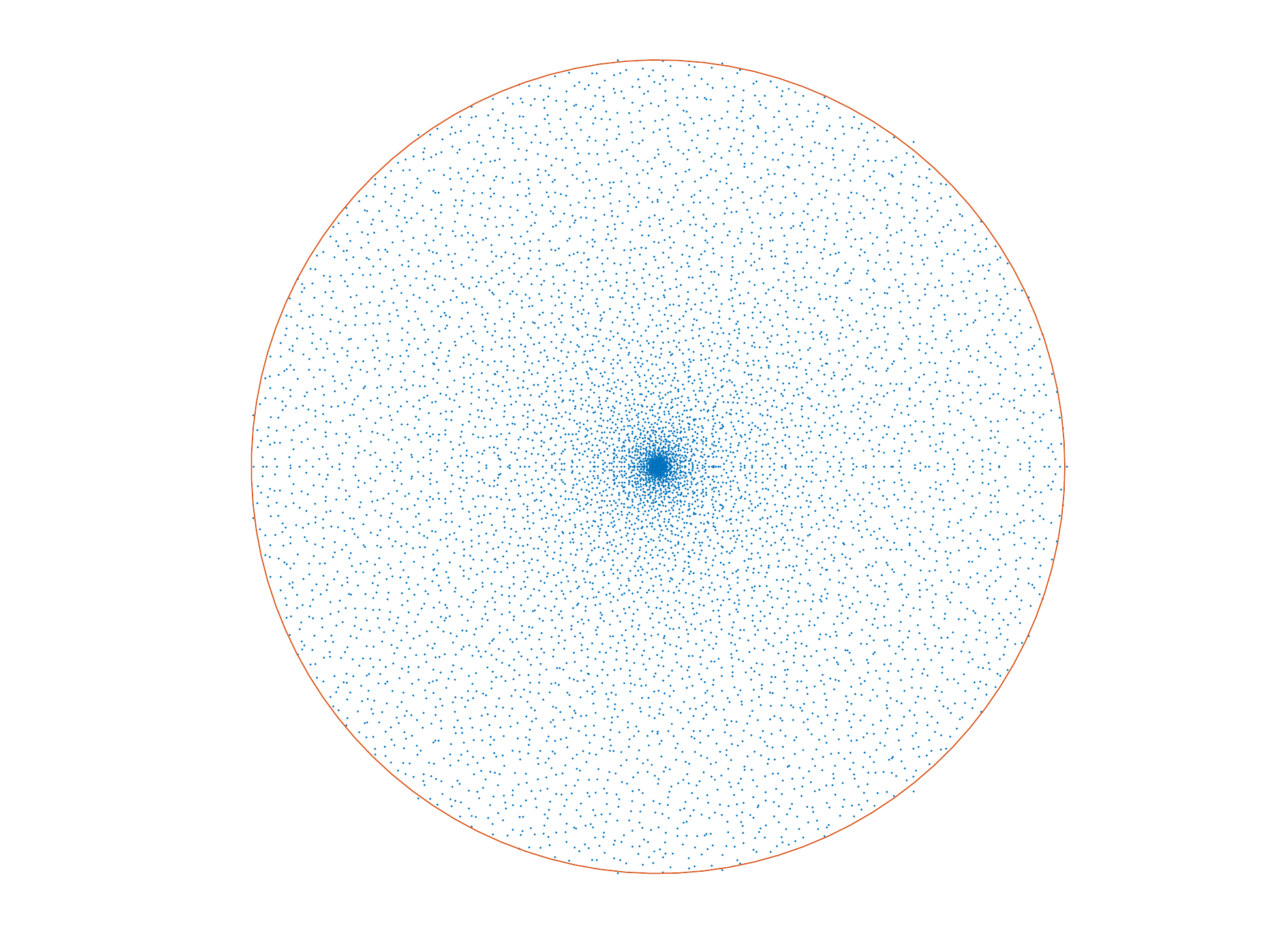}%
        \label{fig:a=1/4}%
        }%
    \caption{Computed eigenvalues (in blue) for $PL/\sqrt{ n \sigma^2{(1-\alpha)} }$ where $n = 2^{14}=16,384$ and $PL \sim PL_n(\xi,\alpha)$ for $\xi \sim N(0,1)$ (where $\sigma^2=1$) and {(a) $\alpha = 0$, (c) $\alpha = 1/4$, and (d) $\alpha=3/4$}, along with (b) $n$ iid samples from $\Uniform(\mathbb D)$, mapped against the unit complex circle $\partial \mathbb D$ (in red)}
    \label{fig:PL fixed sparsity}
    \end{figure}

Using $A \sim \mathcal{PL}_n(\xi,\alpha)$ for fixed {$\alpha \in [0,1)$} and $\xi \sim N(0,1)$, empirical data suggest a similar asymptotic result for the ESD of $A/\sqrt{ n \sigma^2{(1-\alpha)}}$. Note the scaling matches that of \Cref{thm:circular} where {$\alpha = 0$} as well as that seen in \Cref{fig:PL eig} where $\alpha = \frac12$. In particular, following the trajectory for {$\alpha = 0$ in \Cref{fig:Circular}, $\alpha = \frac14$ in \Cref{fig:a=1/4}}, $\alpha = \frac12$ in \Cref{fig:PL eig} (recall $P \sim \Uniform(\mathcal P_n)$ and $P \sim \Uniform(\mathcal P_n^{n\operatorname{-cycles}})$ empirically result in indistinguishable spectral pictures), and {$\alpha = \frac34$} in \Cref{fig:a=3/4}, suggests the scaling by $\sqrt{ n \sigma^2 {(1-\alpha)}}$ have the corresponding ESDs limit to $\nu_\alpha$, a fixed probability measure with support on $\mathbb D$ that depends on $\alpha$ (and not on $\xi$), which further converge to $\nu$, the uniform measure on $\mathbb D$, as {$\alpha \to 0$} and converge to $\delta_0$, the Dirac measure at the origin, as {$\alpha \to 1$}. So the limiting measure is an interpolation  between $\nu$ and $\delta_0$. 

Together, these suggest different universality classes than those included by the Circular law. Previous studies of sparsity and the Circular law have studied iid ensembles that have sparsity $\alpha = \alpha_n$ converging to {1} slow enough whose ESDs still limit to $\nu$ \cite{pmw,rudelson}. Other studies have similarly explored the impact of sparsity on the extreme eigenvalues in the Hermitian case, which has ESDs limiting to the semicircular law \cite{tik_semi}. Results in the literature for fixed sparsity random ensembles remain sparse. The above discussion provides supporting evidence for the following conjecture:

\begin{conjecture}
\label{conj}
    Fix $\alpha \in {[0,1)}$. Let $A = P L Q$ be the $n\times n$ matrix, where $P$ and $Q$ are iid uniformly chosen permutation matrices and $L$ is a $n\times n$ random matrix independent of $P$ and $Q$ whose nonzero entries are iid from $\xi$ with $\E \xi = 0$, $\E|\xi| = \sigma^2$ and $\E|\xi|^4 < \infty$, where $L_{ij} = 0$ if $i+\lfloor k_\alpha \rfloor < j$. Then there exists a probability measure, $\nu_\alpha$, on $\mathbb D$ that is an interpolation between the uniform measure on $\mathbb D$, $\nu$, and the Dirac measure at the origin, $\delta_0$, such that $\mu_{A_n/\sqrt{ n \sigma^2{(1-\alpha)}}}$ converges weakly to $\nu_\alpha$ in probability and almost surely. Furthermore, $\nu_\alpha \to  \nu$ as {$\alpha \to 0$} and $\nu_\alpha \to \delta_0$ as {$\alpha \to 1$}, with both convergences holding uniformly with respect to the total variation distance.
\end{conjecture}
\begin{remark}
For {$\alpha = 0$}, this is the circular law.
\end{remark}
\begin{remark}
Note the right permutation matrix $Q$ term is irrelevant, since $A$ is similar to $QPL$, and $QP \sim P$ since the uniform measure on $S_n$ is left- and right-invariant (it is the Haar measure on $\mathcal P_n$). So the study of the above ensembles reduces to the ensembles of the form $\mathcal{PL}_n(\xi,\alpha)$ for centered $\xi$ with finite fourth moments.
\end{remark}

\begin{remark}
Additional random ensembles that can be studied in light of \Cref{conj} include perturbations of $\mathcal{PL}_n(\xi,\alpha)$ for deterministic matrices (similar to those studied in \cite{pmw,TaoVu}), as well as $P(L + \V I_n)$ for $PL \sim \mathcal{PL}_n(\xi,\alpha)$. This latter model is interesting when $\alpha {>} \frac12$ since the corresponding $L$ factor has eigenvalues of 0 with multiplicity $n$; in particular, $L$ and hence $PL$ does not have full rank. Using experiments for several fixed $\alpha {>} \frac12$, then the nullity of $PL \sim \mathcal{PL}_n(N(0,1),\alpha)$ appears to be approximately $(1 - \sqrt{2{(1-\alpha)}})n$,\footnote{So the rank is approximately $\sqrt{2(1-\alpha)}n$.} while 0 is an eigenvalue of multiplicity approximately $(1 - 2{(1-\alpha)})n{=(2\alpha - 1)n}$; when $\alpha {\le} \frac12$, both are 0 (a.s.). Conversely, now considering $A=P(L + \V I_n)$, then 0 is never an eigenvalue for $A$ and so $A$ is always full rank for all $\alpha$ {(a.s.)}.
\end{remark}

\section{Numerical experiments}
\label{sec:num}

The final section  focuses on a set of experiments that  studies the number of GEPP pivot movements needed on particular random linear systems. These experiments will expand on results first presented in \cite{jpm}, which studied the impact on the growth factors and relative error computations when using common random transformations from numerical linear algebra on particular fixed linear systems, $A \V x = \V b$. Both initial models represent scenarios when no GEPP pivot movements are needed, so we will refer to both here as min-movement models. Carrying forward the naming scheme from \cite{jpm}, the two linear systems studied include:
\begin{enumerate}
\item the \emph{min-movement (na\"ive) model}\footnote{The na\"ive model was named to reflect that using any method is unnecessary to solve the trivial linear system $\V I \V x = \V x= \V b$.}, with $A = \V I_n$ where {the GEPP growth factor} $\rho(\V I_n) = 1$ is minimized, and
\item the \emph{min-movement (worst-case) model}\footnote{The worst-case moniker was chosen to reflect the numerical stability of computed solutions using GEPP, which is controlled by the growth factors: solving $A_n \V x = \V b$ sees relative errors of order $\mathcal O(1)$ starting when $n \approx 60$.}, with $A = A_n$ a particular linear model that maximizes the growth factor $\rho(A_n) = 2^{n-1}$.
\end{enumerate}
In the na\"ive model, the authors studied the 1-sided random transformation $\Omega \V I = \Omega$, which provides a means to directly study the corresponding random matrix, $\Omega$, itself. The worst-case model considered the 2-sided random transformations, $U A_n V^*$, where $U,V$ were independently sampled random matrices; this 2-sided transformation follows the construction used by Parker that would remove the need for pivoting in GEPP (with high probability), so that $UA_n V^* = LU$ has a GE{NP} factorization \cite{Pa95}. The matrix $A_n$ is of the form
\begin{equation}
A_n = \V I_n - \sum_{i > j} \V E_{ij} + \sum_{j=1}^{n-1} \V E_{in}.
\end{equation}
Wilkinson introduced $A_n$ to establish the {GEPP} growth factor bound $\rho(A) \le 2^{n-1}$ is sharp \cite{Wi61}. By construction, no GEPP pivoting would be needed at any intermediate GE step when using $A_n$, so the final GENP and GEPP factorizations of $A_n$ both align, with $A_n = L_nU_n$ for $L_n = \V I_n - \sum_{i>j} \V E_{ij}$ and $U_n = \V I_n - \V E_{nn} + \sum_{k=1}^n 2^{k-1} \V E_{kn}$. It follows $\rho(A_n) = |U_{nn}| = 2^{n-1}$. For example,
\begin{equation}
A_4 = \begin{bmatrix}1&0&0&1\\-1&1&0&1\\-1&-1&1&1\\-1&-1&-1&1 \end{bmatrix} = \begin{bmatrix}1&0&0&0\\-1&1&0&0\\-1&-1&1&0\\-1&-1&-1&1 \end{bmatrix}\begin{bmatrix} 1&0&0&1\\0&1&0&2\\0&0&1&4\\0&0&0&8\end{bmatrix} = L_4U_4
\end{equation}
has $\rho(A_4) = 8 = 2^3$.

With respect to GEPP, both the na\"ive and worst-case models use input matrices that need 0 total pivot movements. We will study both of these min-movement models in addition to a third model, that looks at the other extreme in terms of the number of GEPP pivot movements:
\begin{enumerate}
\setcounter{enumi}{2}
\item the \emph{max-movement model}, where $A \sim \mathcal{PL}^{\max}_n(\xi)$ for $\xi \sim \Uniform([-1,1])$. 
\end{enumerate}
Note if $A = PL \sim \mathcal{PL}_n^{\max}(\xi)$ when $|\xi| < 1$ a.s., then $\rho(A) = \rho(L) = 1$. 

We will consider only the 1-sided transformation case for the na\"ive model (so that we can study the random matrices themselves) and only the 2-sided transformation cases for the worst-case and max-movement models. Together, these three models will allow us to study the impact of random transformations on two different systems where no GEPP pivot movements are needed as well as one (independently sampled random) system where a maximal number of GEPP pivot movements are needed.

For each set of experiments, we will consider the following random transformations for fixed $N = 2^n$:
\begin{itemize}
\item $\B_s(N,\Sigma_S)$
\item $\B(N,\Sigma_S)$
\item $\B_s(N,\Sigma_D)$
\item $\B(N,\Sigma_D)$
\item Walsh transform
\item $\Haar(\O(N))$
\item Discrete Cosine Transform (DCT II)
\end{itemize}
To ease the following discussion, we choose $N = 2^4 = 16$ and $N = 2^8 = 256$ to focus on as we feel they are representative of the behavior we saw for other choices of $N$. For the na\"ive model, which will study the pivot movements for each of the associated random matrices themselves (using the 1-sided preconditioning with $A = \V I$), our experiments will additionally use:
\begin{itemize}
\item $\operatorname{GOE}(N)$
\item $\operatorname{GUE}(N)$
\item $\Bernoulli(\frac12)$
\end{itemize}
These models were touched on for $N = 2$ in \Cref{ex:bern,ex:goe,ex:gue}.

Each of the butterfly models is sampled using custom MATLAB recursive functions with iid uniformly chosen angles\footnote{The number of angles used depends on if the cosine and sine matrices are scalar and if the butterfly matrix is simple. For $\B_s(N,\Sigma_S)$, then one uniform angle is sampled at each recursive step, for $n = \log_2N$ total uniform angles needed, while similarly $\B(N,\Sigma_S)$ and $\B_s(N,\Sigma_D)$ both sample a total of $N - 1$ uniform angles, with $\B(N,\Sigma_D)$ using $\frac12 Nn$ total uniform angles. These compare to $\Haar(\O(N))$, which (using Givens rotations to find the $QR$ factorization of $\operatorname{Gin}(N,N)$) can be sampled using $\binom{N}2=\frac12N(N-1)$ uniform angles. The above ordering reflects this ordering of complexity.} in line with methods outlined in \cite{jpm,Tr19}. See \Cref{subsec:prelim} for more information and sampling techniques of the Walsh, Haar orthogonal, DCT, $\operatorname{GOE}(N)$ and $\operatorname{GUE}(N)$ transformations. The Bernoulli ensemble uses iid $\Bernoulli(\frac12)$ entries\footnote{$\Bernoulli$ matrices are sampled using native MATLAB functions, \texttt{round(rand(n))}.}. Each set of experiments (using $N=2^4$ and $N=2^8$ for all three models) will use 10,000 trials using MATLAB in double {precision}, where $\epsilon = 2^{-52}$ ($\epsilon \approx 2.220446 \cdot 10^{-16}$).










\subsection{Min-movement (na\"ive) model}
For the na\"ive model, our goal is to study the number of GEPP pivot movements needed for 10 different random ensembles. These allow us to gauge the impact on (1-sided) random transformations on the simplest linear system, $\V I \V x = \V b$, in terms of the number of GEPP pivot movements needed. Starting with $\V I$, no pivot movements are needed, so the transformed system $\Omega \V I = \Omega$ then allows us to study how much this transformation introduces new pivot movements. This model also then enables us to directly study the number of pivot movements needed for each random matrix, $\Omega$.

\Cref{t:pivots_naive} shows the sample medians, means ($\bar x$), and standard deviations ($s$) for the 10,000 trials each for $N = 2^4$ and $N = 2^8$, while \Cref{fig:hist_naive} summarizes the total number of GEPP pivot movements encountered for each sampled random matrix across each set of trials. Note the axes in \Cref{fig:hist_naive} show each possible step output for 0 to $N-1$.

\begin{table}[ht!]
\centering
{
\begin{tabular}{r|ccc|ccc}
     &\multicolumn{3}{c}{$N = 16$} &\multicolumn{3}{|c}{$N=256$}\\
     & Median & $\bar x$ & $s$& Median & $\bar x$ & $s$\\ \hline 
$\B_s(N,\Sigma_S)$	&	 8 	 & 	 7.492 	 & 	 1.951 	 & 	 128 	 & 	 127.501 	 & 	 7.978 	 \\ 
$\B(N,\Sigma_S)$	&	 11 	 & 	 10.934 	 & 	 2.622 	 & 	 232 	 & 	 230.392 	 & 	 14.418 	 \\ 
$\B_s(N,\Sigma_D)$	&	 12 	 & 	 11.287 	 & 	 2.459 	 & 	 245 	 & 	 241.915 	 & 	 10.308 	 \\ 
$\B(N,\Sigma_D)$	&	 13 	 & 	 12.535 	 & 	 1.430 	 & 	 250 	 & 	 249.901 	 & 	 2.112 	 \\ 
Walsh	&	 6 	 & 	 6 	 & 	 -   	 & 	 120 	 & 	 120 	 & 	 -   	 \\ 
$\Haar(\O(N))$	&	 13 	 & 	 12.624 	 & 	 1.345 	 & 	 250 	 & 	 249.884 	 & 	 2.123 	 \\ 
DCT II	&	 13 	 & 	 13 	 & 	 -   	 & 	 249 	 & 	 249 	 & 	 -   	 \\ 
$\operatorname{GOE}(N)$	&	 11 	 & 	 10.954 	 & 	 1.780 	 & 	 249 	 & 	 248.696 	 & 	 2.359 	 \\ 
$\operatorname{GUE}(N)$	&	 11 	 & 	 11.132 	 & 	 1.761 	 & 	 249 	 & 	 248.867 	 & 	 2.348 	 \\ 
Bernoulli	&	 11 	 & 	 10.509 	 & 	 1.774 	 & 	 248 	 & 	 247.783 	 & 	 2.467 			
\end{tabular}
}
\caption{Pivot counts for numerical experiments for GEPP with 10,000 trials, for random matrices of orders $N=2^4$ and $N=2^8$}
\label{t:pivots_naive}
\end{table}

\begin{figure}[htbp!]
  \centering
  \includegraphics[width=0.8\textwidth]{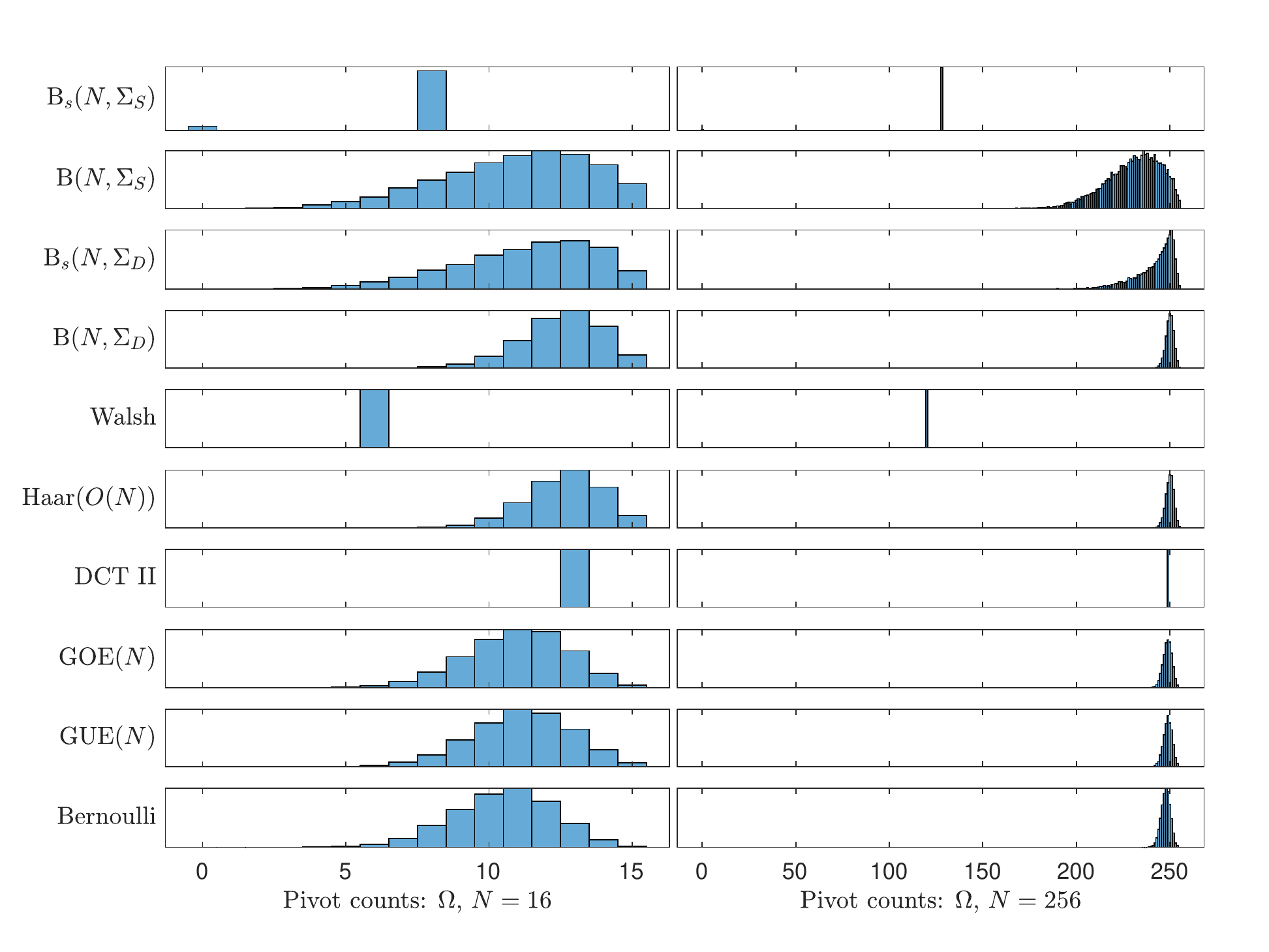}
  \caption{Histogram of $10^4$ samples of pivot movement counts for random matrices of order $N = 2^4$ and $N = 2^8$.}
  \label{fig:hist_naive}
\end{figure}

\subsubsection{Discussion}

For each set of experiments, the Haar-butterfly and Walsh transforms introduce the least amount of additional movement from the initial minimal GEPP movement setup, each with total pivot movements at most $N/2$ while the remaining models had total pivot movements closer to the upper bound of $N-1$.

By construction, both the Walsh and DCT models have deterministic output. This follows since the transformations are of the form $WD$ where $W$ is a deterministic associated matrix (the default Fast-Walsh Hadamard matrix or the DCT matrix used by the native MATLAB functions for the associated multiplication operators) while $D$ is a random row sign matrix sampled uniformly from $\{\pm 1\}^N$. Hence, if the GEPP factorization of $W$ is $PW = LU$, then the GEPP factorization of $WD$ is $PWD = L(UD)$. So the permutation matrix factor is independent of $D$ for these models. This is reflected in \Cref{fig:hist_naive,t:pivots_naive} (e.g., both sample standard deviations are 0).

The two Haar random ensembles studied (viz., $\B_s(N,\Sigma_S)$ and $\Haar(\O(N))$) have the full distribution on the number of GEPP pivot movements determined by \Cref{thm:main}, using also \Cref{cor:haar o}. From \Cref{fig:hist_naive}, these two models also appear to represent relative extreme models among the random ensembles considered in these trials, with the resulting uniform GEPP permutation matrix factor yielding the most pivot movements.

For Haar-butterfly matrices, $B \sim \B_s(N,\Sigma_S)$, we can directly compare the sample statistics against the exact distribution statistics for $Y_N {= \Pi(B)} \sim \frac N2 \Bernoulli(1 - \frac1N)$. We can compute exactly
\begin{align}
    \E Y_N &= \frac{N}2 \left(1-\frac1N\right) \quad \mbox{and}\\
    \sigma_{Y_{N}} &= \frac{N}2 \sqrt{\left(1-\frac1N\right) \cdot \frac1N},
\end{align}
This yields that $\E Y_{16} = 7.5$ and $\sigma_{Y_{16}} \approx  1.93649167$, which align (as expected) with the associated sample mean of 7.492 and sample standard deviation of 1.951 from \Cref{t:pivots_naive} for $N=16$. Similarly, the exact values $\E Y_{256} = 127.5$ and 
 $\sigma_{Y_{256}} \approx  7.98435971$ align with the sample statistics $\bar x = 127.501$ and $s = 7.978$ for $N = 256$. Moreover, as can be seen in \Cref{fig:hist_naive}, the trials resulted only in total GEPP movements of 0 or $\frac N2$, as should be expected for a scaled Bernoulli distribution. This agrees with the result from \cite{tikhomirov} that the computed GEPP permutation matrix factors using floating-point arithmetic and exact arithmetic align with very high probability for $\operatorname{Gin}(n,n)$. Other standard sample statistic comparisons for the Haar-butterfly matrices similarly align, as expected\footnote{For example, the sample medians exactly match the exact medians of $N/2$. Also, we can compare the sample proportion $\hat p_N$ to the population success parameter $p_N = 1 - \frac1N$. These again compare very favorably, where $1 - \hat p_{16} = 0.0635$ aligns with $1 - p_{16} = \frac1{16} = 0.0625$ and $1 - \hat p_{256} = 0.0039$ aligns with $1 - p_{256} = \frac1{256} = 0.00390625$.}.
 
Similarly, we can compare the output for the $\Haar(\O(N))$ trials, which have {$X_N = \Pi(A)$ for $A \sim \Haar(\O(N))$}, the total GEPP pivot movements needed, equal in distribution to $N - \Upsilon_N$. We can compute exactly
\begin{align}
    \E X_N &= N - \E \Upsilon_N = N - H_N \quad \mbox{and}\\
    \sigma_{X_N} &= \sigma_{\Upsilon_N} = \sqrt{H_N - H_N^{(2)}}.
\end{align}
This yields that $\E X_{16} \approx 12.619271006771006$ and  $\sigma_{X_{16}} \approx 1.340291930806123$, which align with the associated sample mean of 12.624 and sample standard deviation of 1.345 from \Cref{t:pivots_naive} for $N = 16$. Similarly, the exact values $\E X_{256} \approx 249.8756550371827$ and $\sigma_{X_{256}} \approx 2.117382706670809$ align with the sample statistics $\bar x = 249.696$ and $s = 2.123$ for $N = 256$.


\Cref{fig:hist_naive} shows the butterfly model {total} pivot movements lie strictly between the pivot movements for Haar-butterfly matrices and $\Haar(\O(N))$, with the increase in associated number of uniform angles needed for the butterfly models leading to the sample distributions progressively moving toward the $\Haar(\O(N))$ model for both $N = 16$ and $N=256$. While $\B(N,\Sigma_D)$ results in pivot movements very close to those modeled by the uniform GEPP permutation matrix factors, $\B(N,\Sigma_S)$ and $\B_s(N,\Sigma_D)$ lie strictly in between both the Haar-butterfly and Haar orthogonal pivot movements. Moreover, the remaining random models for $\operatorname{GOE}(N)$, $\operatorname{GUE}(N)$ and $\Bernoulli$ have pivot movement distributions staying to the left of the Haar orthogonal model, which move closer to the Haar orthogonal model distribution as $N$ increases. This suggests as $N$ increases for these remaining models, the resulting random GEPP permutation matrix moves closer to a uniform permutation matrix. 

\begin{remark}
For both the Haar-butterfly and Haar orthogonal models, the 1-sided na\"ive model is equivalent to the 2-sided na\"ive model since $U\V I_N V^* = UV^* \sim U$ by the right-invariance of the Haar measure. This does not hold, however, for any other random matrices in the na\"ive experiments.
\end{remark}

\begin{remark}For the Bernoulli model, it is possible to generate a singular random matrix, which occurs with probability $2^{-N}(1+o(1))$
\cite{bernoulli}. Of the 10,000 trials, this occurred 48 times for $N = 16$ and 0 times for $N = 256$. \Cref{t:pivots_naive,fig:hist_naive} show the summary statistics and overall GEPP pivot counts for the remaining 9,952 nonsingular Bernoulli matrices when $N = 16$.
\end{remark}

\subsection{Min-movement (worst-case) model}

For the worst-case model, we want to study the number of GEPP pivot movements needed when starting with a fixed linear system, $A_N$, that again requires no pivot movements. This provides a means to measure how much new GEPP pivot movements are generated by these random transformations. For this model, we will consider only the 2-sided transformation, $U A_N V^*$, where $U$ and $V$ are iid samples from each random model used in the experiments.

Analogously to the na\"ive model, \Cref{t:pivots_wc} shows the sample medians, means ($\bar x$), and standard deviations ($s$) for the 10,000 trials again for $N = 2^4$ and $N = 2^8$, while \Cref{fig:hist_wc} summarizes the total number of GEPP pivot movements encountered for each sampled $UA_NV^*$ across each set of trials. 

\begin{table}[ht!]
\centering
{
\begin{tabular}{r|ccc|ccc}
     &\multicolumn{3}{c}{$N = 16$} &\multicolumn{3}{|c}{$N=256$}\\
     & Median & $\bar x$ & $s$& Median & $\bar x$ & $s$\\ \hline 
$\B_s(N,\Sigma_S)$	&	 11 	 & 	 10.563 	 & 	 2.380 	 & 	 179 	 & 	 181.784 	 & 	 27.493 	 \\ 
$\B(N,\Sigma_S)$	&	 12 	 & 	 12.217 	 & 	 1.760 	 & 	 249 	 & 	 247.936 	 & 	 4.322 	 \\ 
$\B_s(N,\Sigma_D)$	&	 12 	 & 	 11.967 	 & 	 1.995 	 & 	 249 	 & 	 247.493 	 & 	 5.521 	 \\ 
$\B(N,\Sigma_D)$	&	 13 	 & 	 12.558 	 & 	 1.406 	 & 	 250 	 & 	 249.876 	 & 	 2.090 	 \\ 
Walsh	&	 12 	 & 	 12.245 	 & 	 1.617 	 & 	 250 	 & 	 249.654 	 & 	 2.253 	 \\ 
$\Haar(\O(N))$	&	 13 	 & 	 12.636 	 & 	 1.335 	 & 	 250 	 & 	 249.900 	 & 	 2.129 	 \\ 
DCT II	&	 12 	 & 	 11.820 	 & 	 1.719 	 & 	 250 	 & 	 249.447 	 & 	 2.313 			
\end{tabular}
}
\caption{Pivot counts for numerical experiments for GEPP with 10,000 trials, for 2-sided transformation of Worst-case model of orders $N=2^4$ and $N=2^8$}
\label{t:pivots_wc}
\end{table}

\begin{figure}[htbp!]
  \centering
  \includegraphics[width=0.8\textwidth]{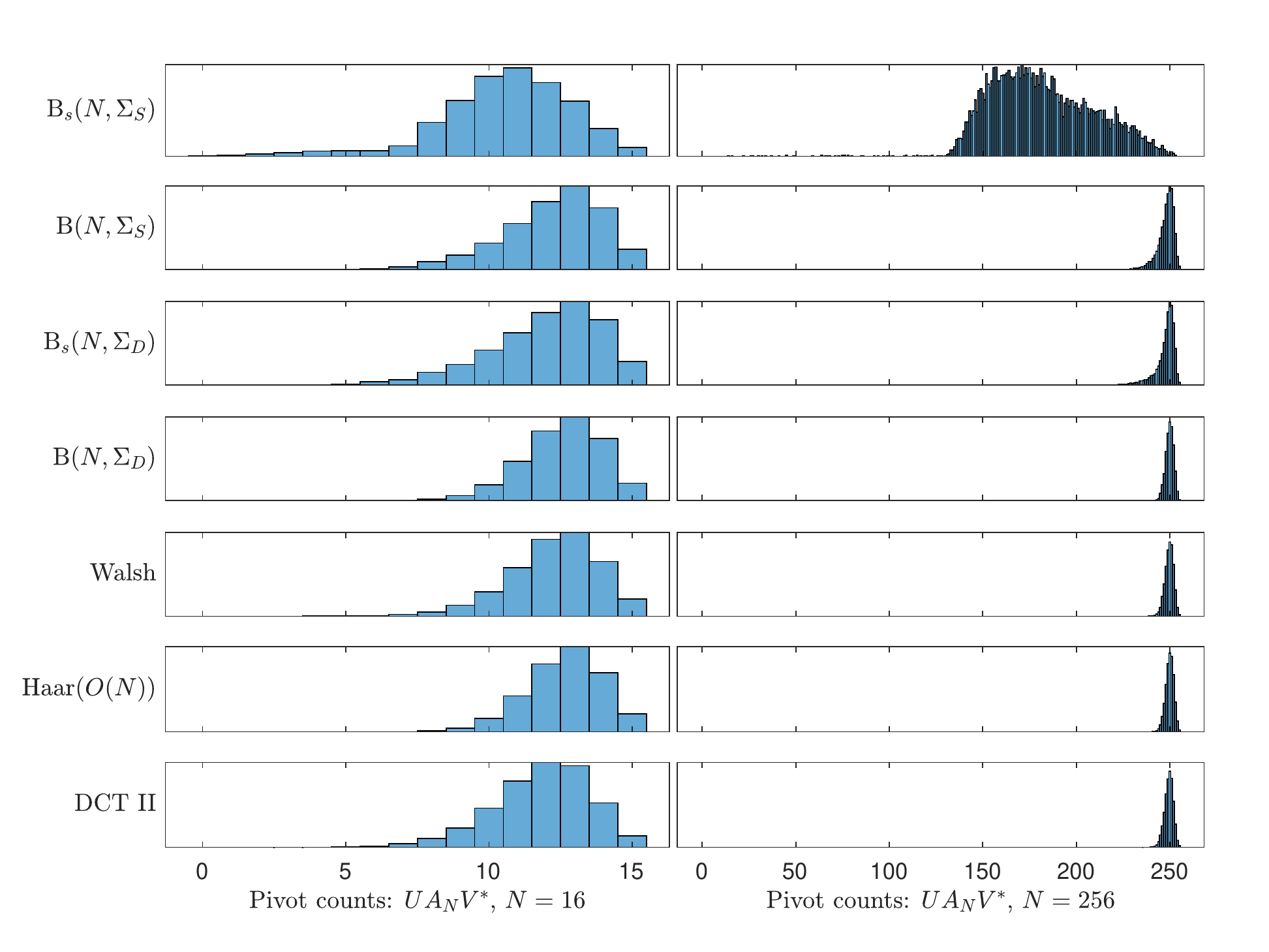}
  \caption{Histogram of $10^4$ samples of pivot movement counts for 2-sided random transformations of order $N = 2^4$ and $N = 2^8$ worst-case model, $UA_N V^*$.}
  \label{fig:hist_wc}
\end{figure}

\subsubsection{Discussion}

Only the $\Haar(\O(N))$ model for $UA_NV^*$ is sampled from a distribution determined in \Cref{thm:main}, since \Cref{cor:2sided} yields {the} resulting GEPP permutation matrix factor is a uniform permutation matrix, so that $X_N{ = \Pi(UA_NV^*)}$, the number of GEPP pivot movements, is equal in distribution to $N - \Upsilon_N$. Since $A_N$ does not preserve the Kronecker product structure, the Haar-butterfly pivot movement distribution is not preserved. Hence, the number of GEPP pivot movements is no longer a scaled Bernoulli distribution and now has full support on $0,1,\ldots,N-1$.

Again, both the Haar-butterfly and Haar orthogonal models provide representatives for the extremes in the number of pivot movements introduced by these random transformations. As in the na\"ive model, the Haar-butterfly transformation introduced the least amount of new GEPP pivot movements for the initial minimal GEPP pivot movement model $A_N$. 

Of the remaining models, only $\B(N,\Sigma_S)$ and $\B_s(N,\Sigma_D)$ have resulting distributions that do not appear to align with the Haar orthogonal model, although they both appear much closer to the Haar orthogonal than Haar-butterfly models. The remaining models' alignment with the Haar orthogonal models manifests even for the small $N = 16$ experiments for $\B(N,\Sigma_D)$ and the Walsh transform: the exact values $\E X_{16} \approx 12.619271006771006$ and  $\sigma_{X_{16}} \approx 1.340291930806123$ compare to the remaining respective samples means of 12.558 and 12.245 and sample standard deviations of 1.406 and 1.617 for the $\B(N,\Sigma_D)$ and Walsh models. This alignment is even more pronounced for $N = 256$:  the exact values $\E X_{256} \approx 249.8756550371827$ and $\sigma_{X_{256}} \approx 2.117382706670809$ line up very well for $\B(N,\Sigma_D)$, Walsh, and DCT II models, whose sample means range from 249.447 to 249.876 and whose sample standard deviations range from 2.090 to 2.313. Moreover, the remaining models have sample medians each of 250 that exactly match that for the Haar orthogonal model for $N = 256$, while the sample medians match or are  smaller by one than the true Haar orthogonal sample median of 13 for $N = 16$. Again, these suggest performance for the non-butterfly models moving toward the uniform row permutation {model} as $N$ increases.

\subsection{Max-movement model}

While the min-movement models studied the impact of random transformations on the number of pivot movements introduced to initial models that require no GEPP pivot movements, the max-movement model will instead study the impact of the random transformations on a model that has maximal GEPP pivot movements, $PL \sim \mathcal{PL}_N^{\max}(\xi)$ for $\xi \sim \Uniform([-1,1])$. (Unlike the min-movements models, the input matrix $PL$ is random.) This provides a means to measure how much GEPP pivot movement can be removed by these random transformations. As in the worst-case model, we will consider only the 2-sided transformation, $U PL V^*$, where $U$ and $V$ are iid samples from each random model.

\Cref{t:pivots_pl} shows the sample medians, means ($\bar x$), and standard deviations ($s$) for the 10,000 trials each for $N = 2^4$ and $N = 2^8$, while \Cref{fig:hist_pl} summarizes the total number of GEPP pivot movements encountered for each sampled matrix $UPLV^*$.

\begin{table}[ht!]
\centering
{
\begin{tabular}{r|ccc|ccc}
     &\multicolumn{3}{c}{$N = 16$} &\multicolumn{3}{|c}{$N=256$}\\
     & Median & $\bar x$ & $s$& Median & $\bar x$ & $s$\\ \hline 
$\B_s(N,\Sigma_S)$	&	 13 	 & 	 12.580 	 & 	 1.348 	 & 	 250 	 & 	 249.864 	 & 	 2.126 	 \\ 
$\B(N,\Sigma_S)$	&	 13 	 & 	 12.594 	 & 	 1.369 	 & 	 250 	 & 	 249.899 	 & 	 2.090 	 \\ 
$\B_s(N,\Sigma_D)$	&	 13 	 & 	 12.613 	 & 	 1.357 	 & 	 250 	 & 	 249.901 	 & 	 2.120 	 \\ 
$\B(N,\Sigma_D)$	&	 13 	 & 	 12.626 	 & 	 1.322 	 & 	 250 	 & 	 249.887 	 & 	 2.121 	 \\ 
Walsh	&	 13 	 & 	 12.630 	 & 	 1.332 	 & 	 250 	 & 	 249.879 	 & 	 2.123 	 \\ 
$\Haar(\O(N))$	&	 13 	 & 	 12.625 	 & 	 1.339 	 & 	 250 	 & 	 249.833 	 & 	 2.130 	 \\ 
DCT II	&	 13 	 & 	 12.573 	 & 	 1.344 	 & 	 250 	 & 	 249.923 	 & 	 2.116 			
\end{tabular}
}
\caption{Pivot counts for numerical experiments for GEPP with 10,000 trials, for 2-sided transformation of max-movement model of orders $N=2^4$ and $N=2^8$}
\label{t:pivots_pl}
\end{table}

\begin{figure}[htbp!]
  \centering
  \includegraphics[width=0.8\textwidth]{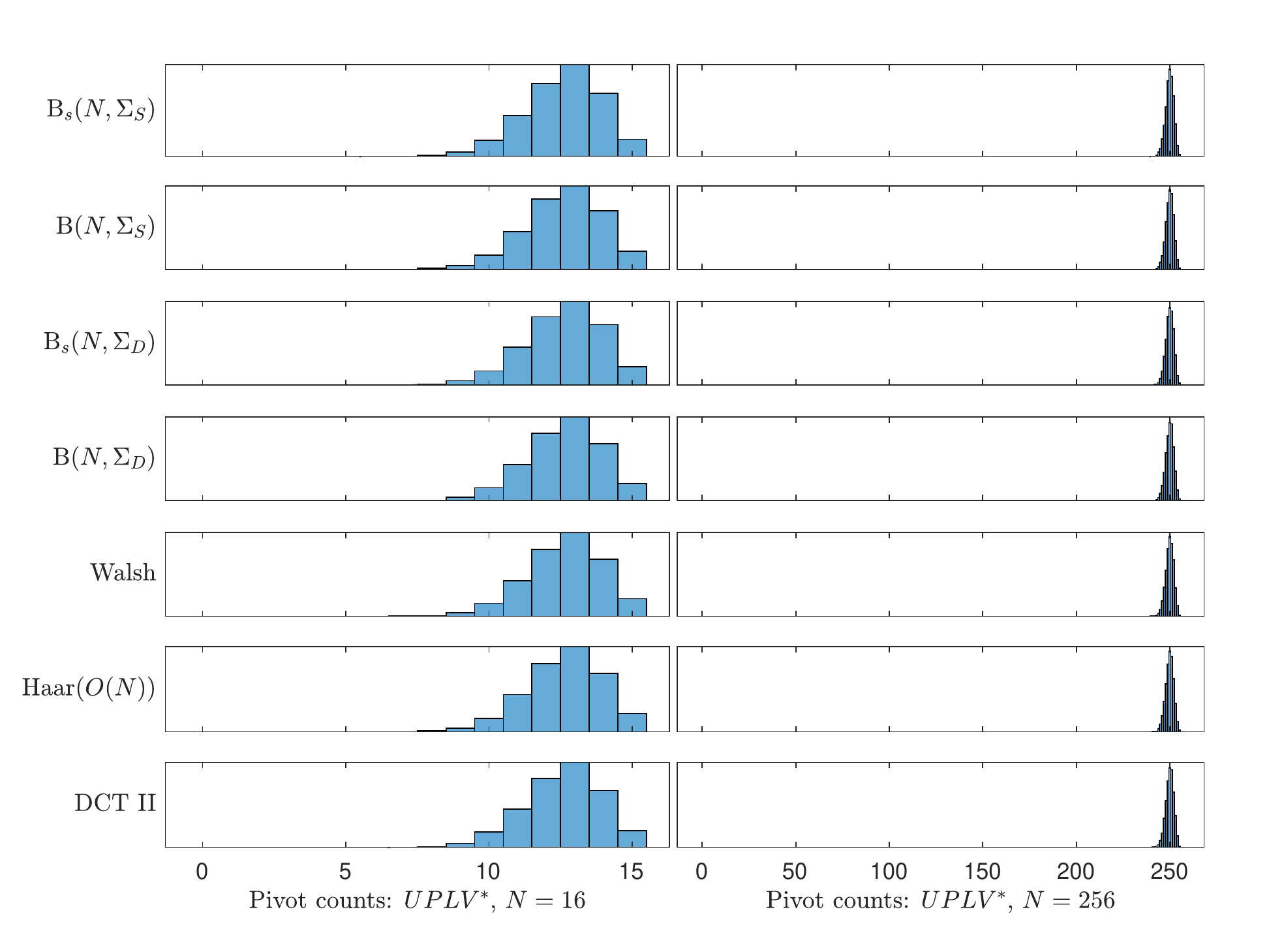}
  \caption{Histogram of $10^4$ samples of pivot movement counts for 2-sided random transformations of order $N = 2^4$ and $N = 2^8$ maximal movement real model, $UPL V^*$.}
  \label{fig:hist_pl}
\end{figure}

\subsubsection{Discussion}

As in the worst-case model, only the Haar orthogonal transformed model $U PL V^*$ has its distribution determined by \Cref{thm:main}, where \Cref{cor:2sided} again yields $X_N{ = \Pi(UPLV^*)}$, the number of GEPP pivot movements, correspond to uniform row permutations, so $X_N \sim N - \Upsilon_N$. Unlike both min-movement models, all of the resulting experiments align strongly with this uniform row permutation model. All of the sample means are within 0.05 of the exact means $\E X_{N}$ and all of the sample standard deviations are within 0.02 of the exact standard deviations $\sigma_{X_N}$ for both $N = 16$ and $N = 256$ (see \Cref{t:pivots_pl}). Moreover, every set of experiments exactly matched the true medians for the uniform row permutation models of 13 for $N = 16$ and 250 for $N = 256$. Hence, this suggests every random transformation had essentially equivalent dampening impacts on the total GEPP pivot movements when starting with a maximal pivot movement model.


\subsection{Conclusions}

The Haar orthogonal model, which had GEPP pivot movements $X_n \sim n - \Upsilon_n$, remains a strong comparison point for each random transformation across each min- and max-movement model. In each case, $X_n$ represents both an upper bound for overall performance in terms of numbers of pivot movements, as well as a limiting class for most random transformations, which further suggest a universality result in terms of GEPP pivot movements. Since the Haar orthogonal model results in uniform GEPP permutation matrix factors, this suggests most random transformation classes have sufficient random mixing properties using both minimal and maximal GEPP movement input models. Undesirably, however, this model asymptotically is concentrated near the upper bound of $n-1$ in terms of total pivot movements, with average $n - H_n = (n-1)(1 + o(1))$. 

The Haar-butterfly model introduced the least amount of additional pivot movements among the min-movement models, while the remaining butterfly models introduced increasing pivot movements as they increased in randomness (i.e., from $\B(N,\Sigma_S)$ and $\B_S(N,\Sigma_D)$ to $\B(N,\Sigma_D)$). However, only the Haar-butterfly models remained far from the upper bound for the min-movement models. In \cite{jpm}, a future direction wanted to explore the impact of combining random transformations (that remove the need of GEPP pivoting) with GEPP on the total number of  pivot movements. To address this concern, these experiments suggest the butterfly models do the least amount of damage in terms of introducing new GEPP pivot movements when starting with a linear system with little GEPP pivot movements necessary. However, no models had strong dampening performance when starting with a max-movement input system.

\begin{appendix}
\section{Notation and preliminaries}
\label{subsec:prelim}

For convenience, $N$ will be reserved for powers of 2, with $N = 2^n$. For $A \in \mathbb F^{n\times m}$ where $\mathbb F = \mathbb R$ or $\mathbb C$,  $A_{ij}$ denotes the entry in the $i$th row and $j$th column of $A$, while $A_{\alpha,\beta}$ will denote the submatrix of $A$ with row indices $\alpha \subset [n] := \{1,2,\ldots,n\}$ and $\beta \subset [m]$. Let $\V e_i$ denote the standard basis elements of $\mathbb F^n$ and $\V E_{ij} = \V e_i\V e_j^T$, the standard basis elements of $\mathbb F^{n\times m}$.   $\V I$ denotes the identity matrix and $\V 0$  the zero matrix or vector (with the dimensions implicit from context if not stated explicitly). If $A \in \mathbb F^{n\times n}$ is nonsingular, then $A^0 := \V I$. Let $\mathbb D = \{z \in \mathbb C: |z| < 1\}$ denote the unit complex disk, with $\partial \mathbb D$ denoting the unit complex circle. We will write $\mathbb S^{n-1} = \{\V x \in \mathbb F^n: \|\V x \|_2 = 1\}$, where $\|\cdot\|_2$ denotes the standard $\ell_2$-norm.

Let $S_n$ denote the symmetric group on $n$ elements. Recall every permutation $\sigma \in S_n$ can be written in cycle notation, with $\sigma = \tau_1 \tau_2 \cdots \tau_j$ where $\tau_i = (a_{i_1} \ a_{i_2} \ \cdots \ a_{i_k})$ is a $k$-cycle, such that $\tau_i(a_{i_m}) = a_{i_{m+1}}$ for $m<k$ and $\tau_i(a_{i_k}) = a_{i_1}$. Moreover, recall every permutation can be written as a product of disjoint cycles. For $\sigma \in S_n$, let $P_\sigma$ denote the orthogonal permutation matrix such that $P_\sigma \V e_i = \V e_{\sigma(i)}$. 
Let $\mathcal P_n$ denote the $n\times n$ permutation matrices, i.e., the left regular representation of the action of $S_n$ on $[n]$. 

Let $\|\cdot\|_{\max}$ denote the element-wise max norm of a matrix defined by $\|A\|_{\max} = \max_{i,j} |A_{ij}|$.  Define $A \oplus B \in \mathbb F^{(n_1+m_1)\times (n_2 + m_2)}$ to be the block diagonal matrix with blocks $A \in \mathbb F^{n_1 \times m_1}$ and $B \in \mathbb F^{n_2\times m_2}$. Define $A \otimes B \in \mathbb F^{n_1n_2 \times m_1m_2}$ to be the {Kronecker product} of $A \in \mathbb R^{n_1 \times m_1}$ and $B \in \mathbb R^{n_2\times m_2}$, given by
\begin{equation}
    \label{eq:kronecker_def}
    A \otimes B = \begin{bmatrix}
    A_{11} B & \cdots & A_{1,m_1} B\\
    \vdots & \ddots & \vdots\\
    A_{n_1,1} B & \cdots & A_{n_1,m_1} B
    \end{bmatrix}.
\end{equation}
Recall Kronecker products satisfy the {mixed-product property}: if all matrix sizes are compatible for the necessary matrix multiplications, then  
\begin{equation}
\label{eq: mixed product}
(A\otimes B)(C \otimes D) = (AC) \otimes (BD),
\end{equation}
i.e., the product of Kronecker products is the Kronecker product of the products. As a result, Kronecker products inherit certain shared properties of their input matrices. For example, if $A$ and $B$ are both orthogonal or unitary matrices, then so is $A \otimes B$. Similarly, if $A \in \mathcal P_n$ and $B \in \mathcal P_m$ then $A \otimes B \in \mathcal P_{nm}$. 

Let $\GL_n(\mathbb F)$ denote the group of nonsingular matrices with entries in $\mathbb F$. Let $\mathcal U_n(\mathbb F)$ denote the subgroup of nonsingular upper triangular matrices and $\mathcal L_n(\mathbb F)$ denote the subgroup of unipotent (i.e., with all diagonal entries equal to 1) lower triangular matrices. $\O(n)$ and $\U(n)$ denotes the orthogonal and unitary groups of $n\times n$ matrices and  $\SO(n),\SU(n)$ denote the respective special orthogonal and special unitary subgroups; note $\O(n)$ will be used for the orthogonal matrices while $\mathcal O(n)$ is the classical ``big-oh'' notation. Recall if $H$ is a subgroup of $G$, then $G/H = \{xH: x \in G\}$ will denote the set of left-cosets of $H$ in $G$ and $G\backslash H = \{Hx: x \in G\}$ the set of right-cosets of $H$ in $G$.

For random variables $X,Y$, we write $X \sim Y$ if $X$ and $Y$ are equal in distribution. Standard distributions that will be used in this document include $X \sim N(0,1)$ to denote a standard Gaussian random variable (with probability density $(2\pi)^{-1/2} e^{-x^2/2}$); $X \sim N_{\mathbb C}(0,1)$ to denote a standard complex Gaussian random variable (with $X \sim (Z_1 + iZ_2)/\sqrt 2$ for $Z_1,Z_2$ iid $N(0,1)$); $X \sim \Uniform(\mathcal A)$ to denote a uniform random variable with support on a compact set $\mathcal A$ with probability density $\frac1{|\mathcal A|}\mathds 1_{\mathcal A}$ (for $|\mathcal A|$ either denoting the cardinality of $\mathcal A$ if $\mathcal A$ is finite or the corresponding appropriate Lebesgue-measure of $\mathcal A$); $\xi \sim \operatorname{Bernoulli}(p)$ to denote a Bernoulli random variable with parameter $p \in [0,1]$ where $\P(\xi = 1) = p = 1-\P(\xi = 0)$; and $\xi \sim \operatorname{Rademacher}$ to denote a Rademacher random variable that takes only the values $1$ and $-1$ with equal probability (i.e., $\xi \sim (-1)^{\operatorname{Bernoulli}(1/2)}$). {A random variable is called (absolutely) continuous if its associated probability measure is absolutely continuous with respect to the Lebesgue measure, where we will assume the standard Lebesgue measure is used if the random variable is real and otherwise the standard complex Lebesgue measure if the random variable is complex valued.} Let $\operatorname{Gin}(n,m)$ denote the $n\times m$ Ginibre ensemble, consisting of random matrices with independent and identically distributed (iid) standard Gaussian entries; $\operatorname{Gin}_{\mathbb C}(n,m)$ will denote the similarly defined complex Ginibre ensemble, whose entries are iid standard complex Gaussian random variables. Let $\operatorname{GOE}(n)$ and $\operatorname{GUE}(n)$ denote the Gaussian Orthogonal and Gaussian Unitary Ensembles, respectively; recall these can be sampled using the Ginibre ensembles as follows: if $G \sim \operatorname{Gin}(n,n)$ and $H \sim \operatorname{Gin}_{\mathbb C}(n,n)$, then $(G+G^T)/\sqrt 2 \sim \operatorname{GOE}(n)$ and  $(H + H^*)/\sqrt 2 \sim \operatorname{GUE}(n)$. 

Let $\epsilon_{\operatorname{machine}}$ denote the machine-epsilon, which is the minimal positive number such that $\operatorname{fl}(1+\epsilon_{\operatorname{machine}} ) \ne  1$ using floating-point arithmetic.\footnote{We will use the IEEE standard model for floating-point arithmetic.} 
If using $t$-bit mantissa precision, then $\epsilon_{\operatorname{machine}} = 2^{-t}$. Our later experiments in \Cref{sec:num} will use double precision in MATLAB, which uses a 52-bit mantissa.

{Standard models from randomized numerical linear algebra will be used for comparison in \Cref{sec:num}. These will include the Walsh transformation and Discrete Cosine Transformations (DCT), which were previously used in \cite{jpm}. Sampling for the following experiments will use native (deterministic) MATLAB functions (viz., the Fast Walsh-Hadamard transform \texttt{fwht} and the default Type II Discrete cosine transform \texttt{dct}) applied after an independent row sign transformation chosen uniformly from $\{\pm 1\}^N$. See \cite{St99,Tr11} for an overview of numerical properties of the Walsh and DCT transforms, and \cite{MaTr20} for a thorough survey that provides proper context for use of these transforms and other tools from randomized numerical linear algebra.}

{Additionally, we will utilize left and right invariance properties of the Haar measure on locally compact Hausdorff topological groups, first established by Weil \cite{We40}. For a compact group $G$, this measure can be normalized to yield a probability measure $\Haar(G)$, which inherits the invariance and regularity properties of the original measure  and yields a means to uniformly sample from compact groups, such as $\O(n)$ and {$\SO(n)$}. Recall every nonsingular matrix $A \in \mathbb F^{n\times n}$ has a $QR$ factorization, with $A = QR$ for $R$ upper triangular with positive diagonal entries and $Q \in \O(n)$ if $\mathbb F = \mathbb R$ or $Q \in \U(n)$ if $\mathbb F = \mathbb C$. Stewart provided an outline to sample from $\Haar(\O(n))$ by using $\operatorname{Gin}(n,n)$ through the $QR$ factorization: if $A \sim \operatorname{Gin}(n,n)$ and $A = QR$ is the $QR$ decomposition of $A$ where $R$ has positive diagonal entries, then $Q \sim \Haar(\O(n))$ \cite{stewart}. Similarly, $\operatorname{Haar}(\U(n))$ can be sampled using $\operatorname{Gin}_{\mathbb C}(n,n)$. Our experiments will employ efficient sampling methods for $\Haar(\O(n))$ that use Gaussian Householder reflectors, in line with the $QR$ factorization of $\operatorname{Gin}(n,n)$ (see \cite{Mezz} for an outline of this method). }

\section{Gaussian elimination and growth factors}
\label{subsec:ge}
GENP iteratively works through the bottom right untriangularized $n-k+1$ dimensional submatrices of the GE transformed matrix $A^{(k)}$ to result in the factorization $A = LU$ for $L$ a unipotent lower triangular matrix and $U$ an upper triangular matrix. $A^{(k)}$ represents the resulting transformed matrix of $A$ at the $k^{th}$ GE step that is zero below the first $k-1$ diagonals and 
\begin{equation}
\label{eq:L}
L_{ij} = \frac{A^{(j)}_{ij}}{A^{(j)}_{jj}}
\end{equation}
for $i>j$, with $A^{(1)} = A$ and {$A^{(n)} = U$}. When GENP \textit{can} be completed (viz., when all leading principal minors are nonzero), the final factorization $A=LU$ can be reused with different input $\V b$ to solve the  computationally simpler triangular systems 
\begin{equation}
\label{eq:ge tri}
L\V y = \V b \quad \mbox{and}  \quad U\V x = \V y.
\end{equation}
Moreover, if $A$ has nonvanishing principal minors, then the resulting $LU$ factorization is unique. See standard references, such as \cite{Hi02}, for an explicit outline of GE.

If GENP cannot be completed, then a pivoting strategy can be applied so that GE can continue at each step, which can involve row or column movements that ensure the leading diagonal entry (i.e., the pivot) of the untriangularized subsystem is nonzero. Different pivoting strategies then result in the modified GE factorization $PAQ = LU$ for $P,Q$ permutation matrices. GEPP remains the most popular pivoting strategy, which uses only row permutations to ensure the leading pivot at the $k^{th}$ GE step is maximal in magnitude among the lower entries in its column. By construction, the $L$ from the resulting GEPP factorization $PA = LU$ satisfies $\|L\|_{\max}  =  1$.  If there is ever a ``tie'' during an intermediate GEPP pivot search, which occurs when $|A_{ij}^{j}| = |A_{jj}^{j}|$ and would result in $|L_{ij}| = 1$ for some $i > j$, then the $L$ and $U$ factors are not unique with respect to row transformed linear systems, i.e., if $A$ has the GEPP factorization $PA = LU$ and $B = QA$ for $Q$ a permutation matrix, then we do not necessarily have the GEPP factorization $(PQ^T) B = LU$. When ties are avoided, GEPP results in unique $L$ and $U$ factors.
\begin{theorem}[\cite{jpm}]
    \label{thm:GEPP_unique}
    Let $A$ be a nonsingular square matrix. Then the $L$ and $U$ factors in the GEPP factorization $PA=LU$ are invariant under row permutations on $A$ iff $|L_{ij}|<1$ for all $i>j$.
\end{theorem}
Moreover, when no ties are encountered with a nonsingular $A$ with $B=QA$ defined as above, then GEPP does necessarily result in the factorization $(PQ^T)B = LU$.

Even when pivoting is not necessary, pivoting can remain desirable for its numerical stability properties when using floating-point arithmetic. Wilkinson first established the backward stability of GEPP by showing the {growth factor},
\begin{equation}
\rho(A) = \frac{\max_k \|A^{(k)}\|_{\max}}{\|A\|_{\max}},
\end{equation}
satisfies the upper exponential bound $1 \le \rho(A) \le 2^{n-1}$ for all matrices $A$ \cite{Wi61}. The growth factor controls the backwards relative error for computed solutions $\mathbf{\hat x}$ using GE, as Wilkinson further established through
\begin{equation}
\frac{\|\mathbf{\hat x} - \V x\|_\infty}{\|\V x\|_\infty} \le 4 n^2  \kappa_\infty(A) \rho(A) \epsilon_{\operatorname{machine}}
\end{equation}
for $\kappa_\infty(A) = \|A\|_\infty \|A^{-1}\|_\infty$  the $\ell_\infty$-condition number. \Cref{sec:num} will consider particular linear models that maximize the GEPP growth factors.

In practice, GEPP implementations result in computed solutions with higher accuracy far from the worst-case exponential behavior Wilkinson's analysis first highlighted. Understanding this behavior remains an important question in numerical analysis. This was partially answered by Huang and Tikhomirov through the use of average-case analysis of GEPP using $A \sim \operatorname{Gin}(n,n)$: they showed  with probability near 1, both the number of bits of precision needed to solve $A\V x = \V b$ to $m$ bits of accuracy is $m + \mathcal O(\log n)$ while also the computed and exact GEPP permutation matrix  factors align \cite{tikhomirov}.

\section{Stirling-1 distribution}
\label{sec:stirling}

This section will delve further into some properties of the Stirling-1 distribution, $\Upsilon_n$, with probability mass function given by \eqref{eq:stirling pmf}. Recall the Stirling numbers of the first kind, $s(n,k)$, arise as the coefficients using the generating function
\begin{equation}
\label{eq:gf stirl1}
(x)_n = \sum_{k=0}^n s(n,k) x^k
\end{equation}
for $(x)_n = x(x-1)\cdots (x-n+1)$, where $s(n,k) = 0$ if not $1\le k \le n$ except $s(0,0) = 1$.\footnote{The notation for Stirling numbers is inconsistent throughout the literature. We are adopting the convention used in \cite{comtet}.} The absolute Stirling numbers of the first kind, $|s(n,k)|$, can similarly be generated using \eqref{eq:gf stirl1} along with $|s(n,k)|=(-1)^{n+k}s(n,k)$; alternatively, $|s(n,k)|$ are determined by the generating function
\begin{equation}
\label{eq:gf abs stirl1}
\langle x\rangle_n = \sum_{k=0}^n |s(n,k)|x^k
\end{equation}
for $\langle x\rangle_n = x(x+1)\cdots (x+n-1)$. \eqref{eq:gf abs stirl1} further establishes the relation
\begin{equation}
\label{eq:stirl1 sym}
|s(n,k)| = s_{n-k}(1,2,\ldots,n-1)
\end{equation}
where 
\begin{equation}
s_j(a_1,a_2,\ldots,a_m) = \sum_{i_1 < \cdots < i_j} \prod_{\ell = 1}^j a_{i_k}
\end{equation}
denotes the elementary symmetric polynomials. This relationship can be used to establish the recurrence
\begin{equation}
\label{eq:stirl rec}
|s(n,k)| = |s(n-1,k-1)| + (n-1) |s(n-1,k)|
\end{equation}
for $k>0$.

Plugging $x=1$ into \eqref{eq:gf abs stirl1} can be used to establish the identity
\begin{equation}
\label{eq:n! id}
n! = \sum_{k=1}^n |s(n,k)|,
\end{equation}
which yields \eqref{eq:stirling pmf} denotes a valid probability density.
An alternative justification for \eqref{eq:n! id} follows from the standard interpretation that $|s(n,k)|$ counts the number of permutations $\sigma \in S_n$ that have exactly $k$ cycles in their disjoint cycle decomposition, where fixed points are counted as $1$-cycles\footnote{This correspondence is justified by noting $\begin{bmatrix} n \\ k \end{bmatrix}$, the number of permutations $\sigma \in S_n$ with $k$ disjoint cycles in their disjoint cycle decomposition, satisfies both the initial conditions along with the recurrence \eqref{eq:stirl rec} (a combinatorial argument yields the analogous recurrence by separating whether 1 comprises its own cycle, which aligns with $|s(n-1,k-1)|$, or if 1 is contained in a larger cycle, which aligns with $(n-1)|s(n-1,k)|$ since there are then $n-1$ places to insert 1 into an existing cycle).}. This interpretation using $S_n$ can be used to provide a combinatorial proof of \eqref{eq:n! id}:  the left hand side of \eqref{eq:n! id} is the number of elements of $S_n$, and the right hand side is the sum of each subset of permutations with a fixed number of $k$ cycles for $k$ ranging from 1 (viz., the $n$-cycles, of which there are $|s(n,1)| = (n-1)!$) to $n$ (viz.,  the identity permutation, in which each object comprises its own cycle of length 1, so that $|s(n,n)| = 1$).\footnote{Similarly, the Stirling numbers of the second kind, $S(n,k)$, can be defined as the number of partitions of $n$ objects into $k$ nonempty sets, which can be connected to $S_n$. $S(n,k)$ and $s(n,k)$ further are related as they yield the coordinates $n,k$ for lower triangular $n\times n$ matrices that are mutual inverses (cf. pg. 144 in \cite{comtet}).} 

Stirling numbers have appeared in connection with statistical problems dating back to their original formulation by Stirling in the 1730s (cf. \cite{ChSi88}). Probabilistic tools have been used to establish and analyze properties of Stirling numbers in the mid- to late-20th century \cite{bellavista,ChSi88,hajime}. $\Upsilon_n$ has appeared as a variant of a more general ensemble of Stirling distributions but has not been studied extensively in the literature. For instance, the mean and variance have been computed for $\Upsilon_n$ (cf. \cite{bellavista}), but general higher moment computations have not been touched. Applying successive derivatives in $x$ to \eqref{eq:gf abs stirl1} and then plugging in $x=1$\footnote{For example, using $$\frac{d}{dx} \langle x \rangle_n = \langle x \rangle_n \cdot \left[ \sum_{j=0}^{n-1} \frac1{x+j}\right] = \sum_{k=1}^n |s(n,k)| k x^{k-1}$$ and then plugging in $x = 1$ yields $n! H_n = \sum_{k=1}^n k |s(n,k)| = n! \mathbb E \Upsilon_n$.} yields
\begin{align}
\label{eq:mean var eta_n}
\mathbb E \Upsilon_n = H_n \quad \mbox{and} \quad
\operatorname{Var} \Upsilon_n = H_n - H_n^{(2)},
\end{align}
where 
\begin{equation}
H_n^{(m)} = \sum_{j=1}^n \frac1{j^m}
\end{equation}
are the generalized Harmonic numbers and $H_n = H_n^{(1)}$ the standard Harmonic numbers. Well known asymptotic results as $n \to \infty$ using Harmonic numbers include $H_n - \log n \to \gamma \approx 0.5772156649$, the Euler–Mascheroni constant, and $H_n^{(m)} \to \zeta(m)$ for $m>1$, where
\begin{equation}
\zeta(m) = \sum_{n\ge 1} \frac1{n^m}
\end{equation}
is the Riemann-zeta function, with $\zeta(2) = \frac{\pi^2}6$.

Continuing with higher moment computations using \eqref{eq:gf abs stirl1} then yields $\E \Upsilon_n^m$ is a function of $H_n^{(j)}$ for $j=1,2,\ldots,m$. Moreover, after computing also the third and fourth moments of $\Upsilon_n$, gathering all resulting terms that use \textit{only} $H_n = H_n^{(1)}$ results in a Bell polynomial of order $m$ with inputs $H_n$ for each component, i.e., 
\begin{equation}
\label{eq:stirl1 moments}
\mathbb E \Upsilon_n \hspace{-.5pc} \mod (H_n^{(2)},H_n^{(3)},\ldots,H_n^{(n)}) = B_n(H_n,\ldots,H_n). 
\end{equation}
This aligns also with the above formulas for $\mathbb E \Upsilon_n = H_n = B_1(H_n)$ and $\mathbb E \Upsilon_n^2 = \operatorname{Var}(\Upsilon_n) +(\mathbb E \Upsilon_n)^2 = H_n^2+H_n - H_n^{(2)} = B_2(H_n,H_n) - H_n^{(2)}$. A future research area can explore finding closed forms for the higher moments of $\Upsilon_n$, including establishing \eqref{eq:stirl1 moments} for higher orders, as well as other general Stirling distributions. 

\section{Butterfly matrices}
\label{subsec:butterfly}

This section will define and introduce relevant background for butterfly matrices and random butterfly matrices, including Haar-butterfly matrices, $\B_s(N,\Sigma_S)$. See \cite{jpm,Tr19} for a fuller discussion of additional numerical and spectral properties of butterfly matrices and random butterfly matrices.

The butterfly matrices of order $N$ are an ensemble of special orthogonal matrices defined recursively as follows: let $\{1\}$ as the order 1 butterfly matrices; for each order $N>1$ butterfly matrix $B$, there exist order $N/2$ butterfly matrices $A_1,A_2$ and order $N/2$ symmetric matrices $C,S$ such that $CS = SC$ and $C^2 + S^2 = \V I_{N/2}$ where
\begin{equation}
B = \begin{bmatrix} C & S\\-S & C \end{bmatrix} \begin{bmatrix} A_1 & \V0\\\V0 & A_2 \end{bmatrix} = \begin{bmatrix} CA_1 & SA_2\\-SA_1 & CA_2 \end{bmatrix}.
\end{equation}
The simple butterfly matrices are formed such that $A_1 = A_2$ at each recursive step. The scalar butterfly matrices, $\B(N)$, are formed using $(C,S) = (\cos \theta \V I, \sin \theta \V I)$ for some angle $\theta$ at each recursive step, where $\B_s(N)$ then denotes the simple scalar butterfly matrices. Note then each $B \in \B(N)$ is of the form 
\begin{equation}
B = (B(\theta) \otimes {\V I_{N/2}}) (A_1 \oplus A_2)
\end{equation}
for
\begin{equation}
B(\theta) = \begin{bmatrix} \cos\theta & \sin\theta\\-\sin\theta & \cos\theta\end{bmatrix}
\end{equation}
the (counter-clockwise) rotation matrix, while each $B \in \B_s(N)$ can then be written of the form
\begin{equation}
B = B(\boldsymbol  \theta) =  \bigotimes_{j=1}^n B(\theta_{n-j+1})
\end{equation}
for $\boldsymbol \theta \in [0,2\pi)^n$. Note $\B_s(N) \subset \B(N) \subset \SO(N)$, with equality when $N \le 2$. While $\B(N)$ is not multiplicatively closed, $\B_s(N)$ forms a closed subgroup of $\SO(N)$ with
\begin{equation}
B(\boldsymbol \theta)B(\boldsymbol \psi) = B(\boldsymbol \theta + \boldsymbol \psi) \quad \mbox{and} \quad B(\boldsymbol \theta)^{-1} = B(-\boldsymbol \theta)
\end{equation}
for $B(\boldsymbol \theta),B(\boldsymbol \psi) \in \B_s(N)$.

Let $\Sigma$ be a collection of dimension $2^k$ pairs $(C_k,S_k)$ of random symmetric matrices with $C_kS_k = S_kC_k$ and $C_k^2+S_k^2 = \V I_{2^{k}}$ for $k\ge 1$. We will write $\B(N,\Sigma)$ and $\B_s(N,\Sigma)$ to denote the ensembles of random butterfly matrices and random simple butterfly matrices formed by  independently sampling $(C,S)$ from $\Sigma$ at each recursive step. Let
\begin{align}
    \Sigma_S &= \{(\cos\theta^{(k)}\V I_{2^{k-1}} ,\sin\theta^{(k)}\V I_{2^{k-1}} ): \theta^{(k)} \mbox{ iid }  \operatorname{Uniform}([0,2\pi), k\ge 1\}
\quad \mbox{and}\\
    \Sigma_D &= \{ \bigoplus_{j=1}^{2^{k-1}}(\cos\t_j^{(k)},\sin\t_j^{(k)}): \theta^{(k)}_j \mbox{ iid }  \operatorname{Uniform}([0,2\pi),  k\ge 1\}.
\end{align}
A large focus for the remainder of this paper is on the {Haar-butterfly matrices}, $\B_s(N,\Sigma_S)$, while numerical experiments in \Cref{sec:num} will also use the other random scalar butterfly ensemble, $\B(N,\Sigma_S)$, along with the random diagonal butterfly ensembles, $\B(N,\Sigma_D)$ and $\B_s(N,\Sigma_D)$. Since $\B_s(N)$ is a compact abelian group, it has a Haar measure that enables uniform sampling of its elements. The name of Haar-butterfly matrices for $\B_s(N,\Sigma_S)$ is precisely because this construction aligns exactly with this Haar measure on $\B_s(N)$.
\begin{proposition}[\cite{Tr19}]
\label{prop:haar butterfly}
$B_s(N,\Sigma_S) \sim \Haar(\B_s(N))$
\end{proposition}

Using the mixed-product property, matrix factorizations of each Kronecker component lead to a matrix factorization of Kronecker products. In particular, this holds for the $LU$ factorizations of $\B_s(N)$ using GENP and GEPP (see \Cref{prop: ge factors}). In particular, the permutation matrix factors from the GEPP factorization of $B \in \B_s(N)$ are from the butterfly permutation matrices, $\mathcal P_N^{(B)}$, which are defined recursively as $\mathcal P_2^{(B)} = \{\V I_2, P_{(1\ 2)}\}$ and
\begin{equation}
\label{eq:def butterfly perm}
\mathcal P_N^{(B)} = \left\{\bigotimes_{j=1}^n P_{(1 \ 2)}^{e_j}: e_j \in \{0,1\} \right\} = \mathcal P_2^{(B)} \otimes P_{N/2}^{(B)}
\end{equation}
for $N > 2$. 
The mixed-product property further yields $\mathcal P_N^{(B)}$ comprises a subgroup of permutation matrices that is isomorphic to $(\mathbb Z / 2\mathbb Z)^n$. Moreover, if $B \sim \B_s(N,\Sigma_S)$, then $P \sim \Haar(\mathcal P_N^{(B)})$ for $PB = LU$ the GEPP factorization of $B$ (see \Cref{cor: haar perm butterfly}).
\end{appendix}

\begin{acks}[Acknowledgments]
The author would like to thank a referee on a previous paper who had asked about the number of movements still needed after using a random transformation on a linear system, which led to the particular direction pursued here. Additionally, the author  thanks  Tom Trogdon, Nick Ercolani, and Adrien Peltzer for many helpful thoughts and insights during the project.
\end{acks}
\bibliographystyle{imsart-number} 
\bibliography{references}       


\end{document}